\documentclass[11pt,a4paper,reqno]{amsart}
\title[C\MakeLowercase{ontractible flow on} $\Stab$ \MakeLowercase{via} $\gldim$]{Contractible flow of stability conditions via global dimension function}
\author{Yu Qiu}
\address{Qy:
	Yau Mathematical Sciences Center and Department of Mathematical Sciences,
	Tsinghua University, 100084 Beijing, China.
    \&
    Beijing Institute of Mathematical Sciences and Applications, Yanqi Lake, Beijing, China}
\email{yu.qiu@bath.edu}
\usepackage{amsfonts}
\usepackage{amsmath,amssymb,amsthm,amsxtra}
\usepackage{float}
\usepackage[colorlinks, linkcolor=blue,anchorcolor=Periwinkle,
    citecolor=red,urlcolor=Emerald]{hyperref}
\usepackage[usenames,dvipsnames]{xcolor}
\usepackage{enumitem}
\usepackage{geometry,array} 
\usepackage{graphicx}
\usepackage{subfigure}
\usepackage{bookmark}
\usepackage{tikz}
\usepackage{tikz-cd}\tikzset{>=latex}
\usepackage{pgfplots}
\usepackage{url}

\usetikzlibrary{matrix,positioning,decorations.markings,arrows,decorations.pathmorphing,
    backgrounds,fit,positioning,shapes.symbols,chains,shadings,fadings,calc}

\tikzset{->-/.style={decoration={  markings,  mark=at position #1 with
    {\arrow{>}}},postaction={decorate}}}
\tikzset{-<-/.style={decoration={  markings,  mark=at position #1 with
    {\arrow{<}}},postaction={decorate}}}

\usepackage[all]{xy}
\usepackage{mathrsfs}
\def\XX{\mathbb{X}}
\def\xx{\mathbf{X}}
\def\ac{\mathbf{A}}
\def\xxc{\mathbf{X}_\mathrm{cd}}

\def\QStab{\operatorname{QStab}}

\newcommand{\nn}{node{$\bullet$}}
\newcommand{\ww}{node[white]{$\bullet$}node{$\circ$}}

\theoremstyle{plain}
\newtheorem{theorem}{Theorem}[section]
\newtheorem*{thm}{Theorem~1}

\newtheorem{lemma}[theorem]{Lemma}
\newtheorem{corollary}[theorem]{Corollary}
\newtheorem{proposition}[theorem]{Proposition}
\newtheorem{conjecture}[theorem]{Conjecture}

\theoremstyle{definition}
\newtheorem{definition}[theorem]{Definition}

\newtheorem{example}[theorem]{Example}

\newtheorem{remark}[theorem]{Remark}
\newtheorem{notations}[theorem]{Notations}
\numberwithin{equation}{section}

\def\hh{\mathcal}

\def\kong{\mathbb}
\def\kk{\mathrm{k}}
\def\<{\langle}
\def\>{\rangle}

\def\ZZ{\mathbb{Z}}

\def\R{\mathbb{R}}
\def\bP{\mathbb{P}}
\def\RR{\R}
\def\bi{\mathbf{i}}

\def\CC{\mathbb{C}}

\def\Aut{\operatorname{Aut}}
\def\Ind{\operatorname{Ind}}
\def\Sim{\operatorname{Sim}}
\def\Hom{\operatorname{Hom}}

\def\Stab{\operatorname{Stab}}
\def\FQuad{\operatorname{FQuad}}
\def\PStab{\operatorname{\mathbb{P}Stab}}

\def\Stap{\operatorname{Stab}^\circ}

\def\diff{\operatorname{d}}

\newcommand{\h}{\hh{H}}            

\renewcommand{\k}{\mathbf{k}}
\renewcommand{\mod}{\operatorname{mod}}

\newcommand{\Cone}{\operatorname{Cone}}

\renewcommand{\Re}{\operatorname{Re}}
\renewcommand{\Im}{\operatorname{Im}}

\newcommand{\id}{\operatorname{id}}

\newcommand{\D}{\operatorname{\hh{D}}}

\newcommand{\per}{\operatorname{per}}

\newcommand{\Tri}{\bigtriangleup}

\def\arrow{red}
\def\surf{\mathbf{S}}                       

\newcommand{\Int}{\operatorname{Int}}

\newcommand\coho[1]{\operatorname{H}^{#1}}

\def\coh{\operatorname{coh}}

\def\T{\kong{T}}

\def\Y{\mathbf{Y}}
\def\P{\mathbf{P}}

\def\jiantou{edge[>=stealth,->]}

\def\gldim{\operatorname{gldim}}
\def\gd{\operatorname{Gd}}

\def\sli{\mathcal{P}}

\def\DQ{\D_\infty(Q)}
\def\DA{\D_\infty(A_n)}
\def\Dpq{\D_\infty(\widetilde{A_{m,r}})}
\def\DS{\D_\infty(\surf^\grad)}

\newcommand{\CP}[1]{\mathfrak{P}(#1)}
\newcommand{\uCP}[1]{\underline{\mathfrak{P}}(#1)}
\def\Poly{\operatorname{Poly}}
\def\grad{\lambda}
\def\num{\mathfrak{N}}
\newcommand{\qv}{\mathbf{q}}
\newcommand{\qInt}{\operatorname{Int}^{\qv}}

\def\VC{\operatorname{\mathbb{V}}}
\def\gms{\surf^\grad}
\begin{document}

\begin{abstract}
We introduce an analytic method that uses the global dimension function $\operatorname{gldim}$ to produce contractible flows on the space $\operatorname{Stab}\mathcal{D}$ of stability conditions on a triangulated category $\mathcal{D}$. In the case when $\mathcal{D}=\mathcal{D}(\mathbf{S}^\lambda)$ is the topological Fukaya category of a graded surface $\mathbf{S}^\lambda$, we show that $\operatorname{gldim}^{-1}(0,y)$ contracts to $\operatorname{gldim}^{-1}(0,x)$ for any $1\le x\le y$, provided $(x,y)$ does not contain `critical' values $\{1+w_\partial/m_\partial \mid w_\partial\ge0, \partial\in\partial\mathbf{S}^\lambda\}$, where the pair $(m_\partial,w_\partial)$ consists of the number $m_\partial$ of marked points and the winding number $w_\partial$ associated to a boundary component $\partial$ of $\mathbf{S}^\lambda$. One consequence is that the global dimension of $\mathcal{D}(\mathbf{S}^\lambda)$ must be one of these critical values.

Besides, we remove the assumptions in Kikuta-Ouchi-Takahashi's classification result on triangulated categories with global dimension less than 1.

    \vskip .3cm
    {\parindent =0pt
    \it Key words:}
    global dimension function, stability conditions, contractible flow, topological Fukaya categories

\end{abstract}
\maketitle

\section{Introduction}
\subsection{Deformation of stability conditions}
The space of stability conditions on a triangulated category, introduced by Bridgeland \cite{B1},
is an interesting homological invariant, which relates representation theory of algebras
and algebraic/symplectic geometry.
Original motivation comes from the study of D-brands in string theory, mirror symmetry,
Donaldson-Thomas theory, etc.
One of the breakthroughs in this direction is the correspondence between
this type of spaces and the moduli spaces of (framed) quadratic differentials,
shown by Bridgeland-Smith \cite{BS} for the Calabi-Yau-3 surface case (cf. \cite{KQ1})
and Haiden-Katzarkov-Kontsevich \cite{HKK} for the Calabi-Yau-$\infty$ surface case.
Aiming to make a precise link between these two works,
we introduce $q$-deformation of categories, stability conditions and quadratic differentials
in the prequels \cite{IQ1,IQ2}.
Namely, for a Calabi-Yau-$\XX$ category $\D_\XX$,
whose Grothendieck group is the $q$-deformation of a rank $n$ lattice,
and any complex number $s$, we identify a subspace $\QStab_s\D_\XX$ (of complex dimension $n$)
in the space $\Stab\D_\XX$ of stability conditions of $\D_\XX$.
We show that one can glue these subspaces $\QStab_s\D_\XX$ under certain conditions
into a complex manifold of dimension $n+1$.
Moreover, $\QStab_s\D_\XX$ can be embedded into the usual spaces of stability conditions
on the corresponding Calabi-Yau-$N$ categories, when $s=N$ is a positive integer.
So the next question is how $\QStab_s\D_\XX$ deforms
when the `Calabi-Yau dimension' $s$ varies,
which will lead to deformation of spaces of stability conditions
along $s$-direction.

From our construction of $\QStab_s\D_\XX$ in \cite{IQ1,IQ2},
the question is closely related to the study of stability condition on Calabi-Yau-$\infty$ categories
(e.g. usual bounded derived categories of algebras or of coherent sheaves on Fano varieties).
One of the key tools here is the global dimension function $\gldim$ (see \cite{Q1,IQ1}).
Our philosophy is that such a function is piecewise Morse and could shed light on
deformation of stability conditions as well as contractibility of spaces of stability conditions
(cf. \cite{FLLQ} for the case of coherent sheaves on the projective plane).

\subsection{Global dimension of triangulated categories}
Global dimension is a classical homological invariant of algebras \cite{A},
or equivalently of their abelian categories.
From 90', triangulated/derived categories become more popular than abelian categories
as they carry more symmetries and are `better' in certain sense.
It is natural to explore the corresponding invariant for triangulated categories
as global dimension for abelian categories.
In \cite{Q1}, we proposed the infimum $\gd\D$ of the global dimension function $\gldim$ on
the space of stability conditions of $\D$ to be a nice candidate as the global dimension for
a triangulated category $\D$ (cf. \cite{Q3}).

In \cite{Q1}, we have shown that
$\gd\DQ=1$ if $Q$ is a non-Dynkin acyclic quiver and
$\gd\DQ=1-2/h_Q$ if $Q$ is a Dynkin quiver,
where $\DQ=\D^b(\mod \kk Q)$, $\kk Q$ is the path algebra of $Q$ and
$h_Q$ the Coxeter number of $Q$ (when $Q$ is a Dynkin quiver).
In \cite{KOT}, Kikuta-Ouchi-Takahashi (KOT) showed that in fact, under some minor condition,
any triangulated category $\D$ with $\gd\D<1$ is equivalent to $\DQ$ for some Dynkin quiver $Q$.
This is a classification theorem of finite type triangulated categories via our global dimension $\gd$,
comparing to the classical version for abelian categories--Gabriel's famous theorem \cite{G}:
\begin{itemize}\item
The module category $\mod \kk Q$ of a quiver $Q$ is of finite type if and only if $Q$ is a Dynkin quiver.
\end{itemize}
In Section~\ref{sec:KOT},
we refine KOT's classification theorem by removing their assumptions,
where the statement becomes (Theorem~\ref{thm:KOT}):
\begin{itemize}\item
any triangulated category $\D$ with $\gd\D<1$ must be of the form
$\DQ/\iota$ for some Dynkin quiver $Q$ and a graph automorphism $\iota$ of $Q$.
\end{itemize}
This is the analogue of Dlab-Ringel's refinement (\cite{DR}) of Gabriel's result.

In Section~\ref{sec:ex}, we calculate global dimensions of graded affine type A quivers
(Theorem~\ref{thm:Apq}) as a first example of non-integer global dimensions
of (non-Calabi-Yau) triangulated categories.

Note that we will actually use this classification result
in the later part of the paper.
\subsection{Test field: topological Fukaya categories}
We mainly focus on topological Fukaya category $\DS$ of a graded marked surface $\gms$ in this paper,
which can be also realized as the bounded derived category of a graded gentle algebra.
There have been a lot of works on this categories, namely,
\begin{itemize}
\item the classification of objects in $\DS$ in \cite{HKK};
\item the description of stability conditions on $\DS$ via quadratic differentials
in \cite{HKK}, cf. \cite{T}.
\item the study of triangle equivalence between different $\DS$ in \cite{LP2};
\item the description on morphisms in $\DS$ in \cite{IQZ} (as a simplified case).
\end{itemize}
Based on these works, we prove the following.
\begin{thm}
Let $\gms$ be a graded marked surface as in Section~\ref{sec:GMS1}.
\begin{itemize}
\item Any stability condition is $\gldim$-reachable (Corollary~\ref{cor:reach2}).
\item If $\gldim\sigma\ge1$, then it equals the maximal angle of the core
of the corresponding quadratic differentials (Proposition~\ref{pp:reach2} and ~\ref{pp:reach1+}).
\item If $1\le\gldim\sigma\notin\VC(\gms)$,
then there is a real submanifold of $\Stap\DS$, open in its closure and, restricted to which,
$\gldim$ is differentiable with no critical point (Theorem~\ref{thm:main}).
\item
If $1\le x<y$ such that $(x,y)\cap\VC(\gms)=\emptyset$,
then $\Stab_{\le y}\DS$ contracts to $\Stab_{\le x}\DS$ (Corollary~\ref{cor:main}).
\item $\gd\DS$ is in $\VC(\gms)$ (Corollary~\ref{cor:gd}).
\end{itemize}
Here $\VC(\gms)=\{1+w_\partial/m_\partial \mid \partial\subset\partial\surf, w_\partial\ge0\}$
is the set of critical values,
where the pair $(m_\partial,w_\partial)$ consist of the number $m_\partial$ of marked points
and the winding number $w_\partial$ associated to a boundary component $\partial$ of $\gms$.
\end{thm}
\subsection*{Acknowledgments}
Qy would like to thank Yu Zhou for many helpful discussion
during collaboration on couple of related papers.
This work is supported by
National Key R\&D Program of China (No. 2020YFA0713000),
Beijing Natural Science Foundation (Z180003).

\section{Preliminaries}\label{sec:pre}
\subsection{Global dimension function of stability conditions}\label{sec:gldim}
Following Bridgeland \cite{B1}, we recall the notion of stability conditions
on triangulated categories.

Throughout the paper, $\D$ is a triangulated category with Grothendieck group $K(\D)\cong\ZZ^n$ for some integer $n$.
Denote by $\Ind\D$ the set of (isomorphism classes of) indecomposable objects in $\D$.
Let $\kk$ be an algebraically closed field.

\begin{definition}
\label{def:stab}
A {\it stability condition} $\sigma = (Z, \sli)$ on $\D$ consists of
a group homomorphism $Z \colon K(\D) \to \CC$, called the {\it central charge}, and
a family of full additive subcategories $\sli (\phi) \subset \D$ for $\phi \in \R$,
called the {\it slicing}, satisfying the following conditions:
\begin{itemize}
\item[(a)]
if  $0 \neq E \in \sli(\phi)$,
then $Z(E) = m(E) e^{\bi\pi \phi}$ for some $m(E) \in \R_{>0}$,
\item[(b)]
for all $\phi \in \R$, $\sli(\phi + 1) = \sli(\phi)[1]$,
\item[(c)]if $\phi_1 > \phi_2$ and $A_i \in \sli(\phi_i)\,(i =1,2)$,
then $\Hom(A_1,A_2) = 0$,
\item[(d)]for $0 \neq E \in \D$, there is a finite sequence of real numbers
\begin{equation}\label{eq:>}
\phi_1 > \phi_2 > \cdots > \phi_l
\end{equation}
and a collection of exact triangles (known as the \emph{HN-filtration})
\begin{equation*}
0 =
\xymatrix @C=5mm{
 E_0 \ar[rr]   &&  E_1 \ar[dl] \ar[rr] && E_2 \ar[dl]
 \ar[r] & \dots  \ar[r] & E_{l-1} \ar[rr] && E_l \ar[dl] \\
& A_1 \ar@{-->}[ul] && A_2 \ar@{-->}[ul] &&&& A_l \ar@{-->}[ul]
}
= E
\end{equation*}
with $A_i \in \sli(\phi_i)$ for all $i$.
\end{itemize}
Nonzero objects in $\sli(\phi)$ are called {\it semistable of phase $\phi$} and simple objects
in $\sli(\phi)$ are called {\it stable of phase $\phi$}.
For semistable object $E\in\sli(\phi)$,
denote by $\phi_\sigma(E)=\phi$ its \emph{phase}.
For any object $E$, define its upper/lower phases
\[
    \phi_\sigma^+(E)=\phi_1,\quad \phi_\sigma^-(E)=\phi_l
\]
via the HN-filtration, respectively.
\end{definition}

In this paper, we will always assume that stability condition satisfies the
technical condition, known as the \emph{support property}, see e.g. \cite{IQ1} for more details.
There is a natural $\CC$-action on the set $\Stab(\D)$ of all stability conditions on $\D$, namely:
\[
    s \cdot (Z,\hh{P})=(Z \cdot e^{-\bi \pi s},\hh{P}_{\Re(s)}),
\]
where $\hh{P}_x(\phi)=\hh{P}(\phi+x)$.
Any auto-equivalence $\Phi\in\Aut(\D)$ also acts naturally on $\Stab(\D)$ as
$$\Phi  (Z,\hh{P})=\big(Z \circ \Phi^{-1}, \Phi (\hh{P}) \big).$$

Recall Bridgeland's key result \cite{B1}, that
$\Stab\D$ is a complex manifold with local homeomorphism
\begin{gather}\label{eq:Z}
    \mathcal{Z} \colon \Stab\D \longrightarrow \Hom_{\ZZ}(K(\D),\CC),
        \quad (Z,\sli) \mapsto Z.
\end{gather}

\begin{definition}
Given a slicing $\hh{P}$ on a triangulated category $\D$.
Define the \emph{global dimension} of $\hh{P}$ by
\begin{gather}\label{eq:geq}
\gldim\hh{P}=\sup\{ \phi_2-\phi_1 \mid
    \Hom(\hh{P}(\phi_1),\hh{P}(\phi_2))\neq0\}\in\RR_{\ge0}\cup\{+\infty\}.
\end{gather}
The global dimension of a stability condition $\sigma=(Z,\hh{P})$ is defined to be $\gldim\hh{P}$.
The \emph{global dimensio}n $\gd\D$ of $\D$ is defined as
\[\gd\D\colon=\inf\,\gldim\Stab\D.\]
\end{definition}

We say $\hh{P}$ (or $\sigma)$ is \emph{$\gldim$-reachable} if
there exist $\phi_1$ and $\phi_2$ such that
\[
    \Hom(\hh{P}(\phi_1),\hh{P}(\phi_2))\neq0
    \quad\text{and}\quad \gldim\hh{P}=\phi_2-\phi_1.
\]
We say $\D$ is \emph{$\gldim$-reachable} if there exists $\sigma$ such that
$\gldim\sigma=\gd\D$.
Note that it is possible that $\Stab\D=\emptyset$ and
then $\gd\D$ is not defined.

\begin{example}
By \cite{Q1,KOT}, we have the following:
\begin{itemize}
\item If $\D=\DQ$ is the bounded derived category of the path algebra of an acyclic quiver $Q$,
then $\D$ is $\gldim$-reachable.
\item If $\D=\D^b(\coh X)$ is the bounded derived category of the coherent sheaves on
a smooth projective curve $X$ of genus $g$ (over $\CC$), then $\gd\D=1$ and
\begin{itemize}
  \item $\D$ is $\gldim$-reachable if $g=0,1$;
  \item $\D$ is not $\gldim$-reachable if $g>1$.
\end{itemize}
\end{itemize}
\end{example}

In \cite{IQ1}, we have shown that $\gldim$ is a continuous function on $\Stab\D$,
which is invariant under the $\CC$-action and $\Aut\D$.

\begin{notations}
Let
$$\Stab_{I}\D\colon=\Stab\D\cap\gldim^{-1}(I)$$
for any $I\subset\RR$.
\end{notations}

A stability condition $\sigma$ on $\D$ is {\em totally (semi)stable}
if every indecomposable object is (semi)stable with respect to $\sigma$.
Note that $\Stab_{\le1}\D$ consists of precisely all totally semistable stability conditions,
and $\Stab_{<1}\D$ consists of all totally stable stability conditions which are $\gldim$-reachable.
(\cite[Prop.~3.5]{Q1}).
\subsection{Type A example}\label{sec:An}
Let us describe all totally stable stability conditions for type A quiver and
give explicit formula of global dimension function in such a case.
Denote by $\Poly(n+1)$ the moduli space of convex $(n+1)$-gon $\P\subset\CC$,
where the vertices $\{V_i\in\CC\mid 0\le i\le n\}$ of the polygons are labelled
in anticlockwise order with $V_0=0$ and $V_1=1$.
The local coordinate of a polygon $\mathbf{P}$ in $\Poly(n+1)$ is given by
its vertices $V_i\in\CC$ for $2\le i\le n$.

Consider the $A_n$ quiver with straight orientation
\begin{gather}\label{eq:An}
    Q=A_n\colon 1\leftarrow 2\leftarrow \cdots \leftarrow n.
\end{gather}
Denote by $\D_\infty(A_n)$ its bounded derived category.
By abuse of notation, let $P_j$ be the projective $\kk Q$ module at $j$.
Denote by $M_{ij}=\Cone(P_{i-1} \to P_j)$ for $1\le i\le j\le n$ (where we set $P_0=0$).

\begin{proposition}\label{pp:A}\cite[Prop.~3.6]{Q1}
There is a natural bijection $\mathfrak{Z}\colon\Stab_{<1}\D_\infty(A_n)/\CC\to\Poly(n+1)$,
sending a stability condition $\sigma$ to a $(n+1)$-gon $\mathbf{P}_\sigma$
such that the oriented diagonals $\overrightarrow{V_{i-1}V_j}$ of $\mathbf{P}_\sigma$
gives the central charges $Z(M_{ij})$
of indecomposable objects in $\D_\infty(A_n)$.
\end{proposition}
\begin{figure}[ht]\centering
\begin{tikzpicture}[xscale=.7,yscale=.5]
\path (-1,0) coordinate (t1) node[below]{$V_0$};
\path (1,0) coordinate (t2) node[below]{$V_1$};
\path (3,3) coordinate (t3) node[right]{$V_2$};
\path (3.5,7) coordinate (t4)node[right]{$V_3$};
\path (-1,10) coordinate (t5)node[above]{$V_4$};
\path (-4,5) coordinate (t6)node[left]{$V_5$};
\draw(8,5)node{$Z(M_{ij})=\overrightarrow{V_{i-1}V_j}$};
\draw(8,4)node[below]{$1\le i\le j\le 5$};
\draw[red](8,1)node[below]{$\alpha=\arg\overrightarrow{V_3V_0}-\arg\overrightarrow{V_1V_4}$};
\foreach \j/\i in {1/4,2/5}
{\draw[red] (t\j) to (t\i);}
\foreach \j in {1,2,3,4,5,6}
{\draw (t\j) node {$\bullet$};}
\foreach \j/\i in {1/2,2/3,3/4,4/5,5/6}
{\draw[->,>=stealth,very thick,blue!50] (t\j) to (t\i);}
\draw[dashed,thick,blue!50] (t6) to (t1);
\draw[blue!50]($(t1)!.5!(t2)$) node[below] {$Z_1$};
\draw[blue!50]($(t3)!.5!(t2)$) node[below right] {$Z_2$};
\draw[blue!50]($(t3)!.5!(t4)$) node[right] {$Z_3$};
\draw[blue!50]($(t5)!.5!(t4)$) node[above] {$Z_4$};
\draw[blue!50]($(t5)!.5!(t6)$) node[above left] {$Z_5$};

\draw[red,->,>=stealth] ($(t1)!.3!(t4)$) to[bend left] ($(t2)!.3!(t5)$);
\draw[red]($(t1)!.25!(t5)$) node[right]{$\quad\alpha$};
\end{tikzpicture}
\caption{Convex hexagon for a totally stable stability conditions on $\D_\infty(A_5)$}\label{fig:tt}
\end{figure}
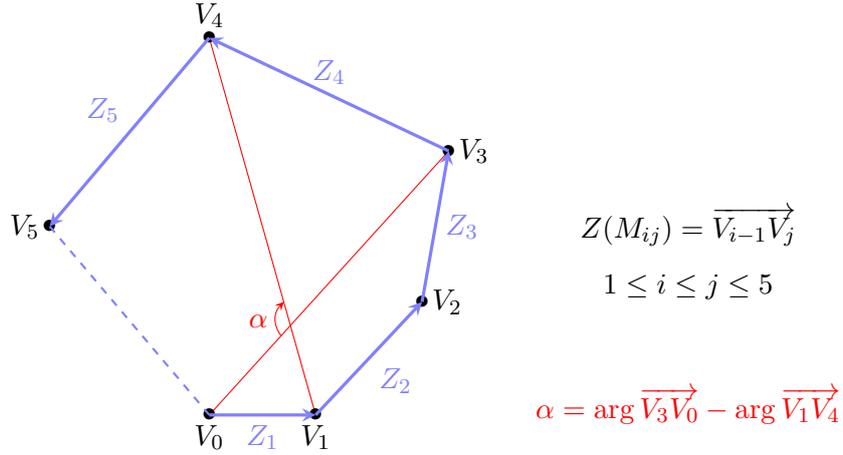

More precisely,
let $\CC\cdot\sigma\in\Stab_{<1}\D(A_n)/\CC$ with representative $\sigma$ such that $Z(P_1)=1$.
Let $\mathbf{P}_\sigma=\mathfrak{Z}(\CC\cdot\sigma)$ be the corresponding $(n+1)$-gon so that
$V_i=Z(P_i)$ for $1\le i\le n$ and then
\begin{gather}\label{eq:gldim a}
    \gldim\sigma=\frac{1}{\pi}\max\{
    \arg\overrightarrow{V_jV_i}-\arg\overrightarrow{V_{i+1}V_{j+1}}
        \mid 0\le i<j\leq n \},
\end{gather}
where $V_{n+1}=V_0$ (cf. Figure~\ref{fig:tt}).

\section{Classification of finite type categories after KOT}\label{sec:KOT}
For an acyclic quiver $Q$, denote by $\DQ$ the bounded derived category of the path algebra $\kk Q$.
Similarly when $Q$ is a specie, cf. \cite{CQ} for details.
Note that any Dynkin specie can be folded from a Dykin quiver.

Let $h_Q$ be the Coxeter number associated to a Dynkin diagram $Q$.
Recall the following, which is a combination of \cite[Thm.~4.7]{Q1} for the quiver case
and \cite[Cor.~6.5]{CQ} for the specie case.

\begin{theorem}\label{thm:Q}
$\gd\DQ=1-2/h_Q$ for a Dynkin quiver or specie $Q$,
where the minimal value of $\gldim$ on $\Stab\DQ$ is
given by the solution of the Gepner equation $\tau\cdot\sigma=(-2/h_Q)\cdot\sigma$.
Moreover, the solution of $\tau\cdot\sigma=(-2/h_Q)\cdot\sigma$ is unique up to $\CC$-action.
\end{theorem}

If $\gd\D<1$, we have the classification theorem for $\D$ (Theorem~\ref{thm:KOT} below).
This is essentially due to Kikuta-Ouchi-Takahashi \cite[Theorem~5.12]{KOT},
where we are going to remove the assumption there:
\begin{itemize}
\item the category $\D$ is the perfect derived category $\per A$ of
some smooth proper differential graded (dg) $\CC$-algebra $A$.
\end{itemize}

Recall the following notions.
\begin{itemize}
\item An object $E$ in $\D$ is \emph{exceptional} if $\Hom^\bullet(E, E)=\kk$.
\item An \emph{exceptional sequence} $\<E_1,\ldots, E_m\>$ in $\D$ is a collection of exceptional objects
such that $\Hom^\bullet(E_i, E_j)=0$ for any $i>j$.
\item An exceptional sequence is \emph{strong} if in addition that $\Hom^k(E_i,E_j)=0$ for any $i,j$ and $k\ne0$.
\item An exceptional sequence is \emph{full}
    if the smallest full triangulated subcategory of $\D$ containing $\{E_i\}$ coincides with $\D$.
\end{itemize}

\begin{theorem}\label{thm:KOT}
Let $\D$ be a connected triangulated category.
Then $\gd\D<1$ if and only if $\D=\DQ/\iota$ for some Dynkin quiver $Q$
and some $\iota\in\Aut\DQ$ induced from some graph automorphism of $Q$.
\end{theorem}
\begin{proof}
By Theorem~\ref{thm:Q}, we only need to show that when $\D$
admits a stability condition $\sigma=(Z,\sli)$ with $\gldim\sigma<1$, then $\D$ must be of Dynkin type as stated.

First we remove the condition that the category $\D$ is over $\CC$
but still assuming it is the perfect derived category $\per A$ of
some smooth proper differential graded (dg) $\kk'$-algebra $A$ over some field $\kk'$
(which is not necessarily algebraically closed).
Then applying the argument in \cite[\S~5.1]{KOT}, we deduce that
$\D$ is locally finite.
By \cite{XZ}, the Auslander Reiten quiver of such a locally finite triangulated category $\D$ is
isomorphic to the orbit $\ZZ Q/\iota$,
where $\ZZ Q$ is the translation quiver of some Dynkin quiver $Q$ and
$\iota$ is an automorphism of $\ZZ Q$.
Note that $\Aut\ZZ Q$ is generated by $[1], \tau$ and graph automorphisms (if exists) of $Q$.
If $\iota^r=[N]$ for some $N\in\ZZ_{>0}$ and $r\in\ZZ$,
then
\[\sli(\ge0)=\sli(\ge0)[N]=\sli(\ge N)\subset\sli(\ge 1)\subset\sli(\ge0).\]
Thus $\sli(\ge0)=\sli(\ge1)$ or $\sli=\sli[1],$ which is a contradiction.
Therefore $\ZZ \iota\cap \ZZ[1]=\emptyset$.
Noticing that $\tau^h_Q=[-2]$, we deduce that $\iota$ can only be an automorphism of/induced by $Q$.
So $\D$ must be of the form $\DQ/\iota$ as required (cf. \cite[Example~1.1]{CQ}).

Next, let us remove all constrains, only assuming that $\D$ is some triangulated category.
We still follow \cite[\S~5.1]{KOT}.
Let $$S(I)\colon=\{\phi\in I\mid \sli(\phi)\ne0\}.$$
If $S(0,1]$ is an infinite set, then
we can take a monotone increasing (similar for decreasing) sequence
$$\phi-\epsilon<\phi_1<\phi_2<\cdots<\phi_m<\cdots<\phi$$
such that $\lim_{m\to\infty} \phi_k=\phi$ and $0<\epsilon<1-\gldim\sigma$.
Let $E_k\in\hh{P}(\phi_k)$.
For any integer $i,j,m\ge1$, we have
\[\begin{cases}
    \phi_i+m>\phi_j,\\
    \gldim\sigma<1-\epsilon<(\phi_j+m)-\phi_i,
\end{cases}\]
which implies
\[\begin{cases}
    \Hom(E_i[m],E_j)=0,\\
    \Hom(E_i,E_j[m])=0,
\end{cases}\]
i.e. $\Hom^\bullet(E_i,E_j)=\Hom(E_i,E_j)$.
If in addition $i>j$, we also have $\Hom(E_i,E_j)=0$.
So $\<E_1,\ldots,E_m\>$ is a full strongly exceptional sequence
in the full thick subcategory $\D^{(m)}\subset\D$ they generated.
This subcategory falls into the case above, i.e. $\D^{(m)}$ is of the form $\D_\infty(Q^{(m)})/\iota$,
where $m$ is the rank/number of vertices of some Dynkin quiver $Q^{(m)}$.
Restricted $\sigma$ to $\D^{(m)}$, we have (\cite[Prop.~5.2]{KOT})
\[
    \gldim\sigma\geq\gldim\sigma|_{\D^{(m)}}\geq 1-2/h_{Q^{(m)}}.
\]
But \[\lim_{m\to\infty} 1-2/h_{Q^{(m)}} = 1,\]
which contradicts to $\gldim\sigma<1$.
Thus $S(0,1]$ is a finite set.

Then we deduce $\D$ is locally finite as in \cite[\S~5.1]{KOT}
and finish the proof as the previous case.
\end{proof}

\section{Topological Fukaya categories}\label{sec:GMS}
\subsection{Graded marked surface}\label{sec:GMS1}
\def\MTS{\mathbb{R}T\surf}
\def\PTS{\mathbb{P}T\surf}
In this subsection, we partially follow \cite[\S~2]{IQZ}, cf. \cite{HKK,LP2}.
A \emph{graded marked surface} $\gms=(\surf,\Y,\grad)$ consists :
\begin{itemize}
\item a smooth oriented surface $\surf$;
\item a set of closed marked points $\Y$ in $\partial\surf$  (cf. \cite[\S~6.2]{IQZ}),
such that $\Y\cap\partial_i\neq\emptyset$ for each boundary component $\partial_i$ of $\partial\surf$.
\item a \emph{grading/foliation} $\grad$ on $\surf$, that is, a section of the projectivized tangent bundle $\PTS$.
\end{itemize}
Let $b=|\partial\surf|$ and $\aleph=|\Y|$.
Then $\partial\surf$ is divided into $\aleph$ many boundary arcs.
The rank of $\surf$ is
\begin{gather}\label{eq:rank}
    n=2g+b+\aleph-2.
\end{gather}
We will require $n\ge2$ to exclude the trivial case.
Denote by $\surf^\circ\colon=\surf\setminus\partial\surf$ its interior.

Note that the projection $\PTS\to\surf$ with $\R\mathbb{P}^1\simeq \mathrm{S}^1$-fiber leads to
a short exact sequence (cf. \cite[\S~2.1]{IQZ})
\[
    0\to \coho{1}(\surf) \to \coho{1}(\PTS) \xrightarrow{\pi_\surf} \coho{1}(\mathrm{S}^1)=\ZZ \to0.
\]
In fact, \cite[Lem.~1.2]{LP2} shows that $\grad$ is determined by a class in $\coho{1}(\PTS)$, denoted by $[\grad]$.
Moreover, such a data $\grad$ is equivalent to a $\ZZ$-covering
$$\operatorname{cov}\colon \MTS^\grad\to\PTS,$$
known as the \emph{Maslov covering},
where $\MTS^\grad$ is the $\R$-bundle of $\surf$ that can be constructed via gluing $\ZZ$ copies of $\PTS$ cut by $\grad$.

A morphism $(f,\widetilde{d f})\colon \gms \to \surf_1^{\grad_1}$ between two graded marked surfaces is
a map $f\colon\surf\to\surf_1$ such that it preserves the marked points and $[\grad]=f^*[\grad_1]$,
regarding $[\grad]\in \coho{1}(\PTS)$, together with
a map $\widetilde{d f}\colon\MTS^\grad \to \mathbb{R}T\surf_1^{\grad_1}$ that fits into the commutative diagram
\[\xymatrix{
    \MTS^\grad \ar[r]^{\widetilde{d f}} \ar[d]_{\operatorname{cov}} & \mathbb{R}T\surf_1^{\grad_1}\ar[d]^{\operatorname{cov}_1}\\
    \PTS \ar[r]^{d f } & \mathbb{P}T\surf_1.
}\]
There is a natural automorphism $[1]$, known as the grading shift on $\gms$,
given by the deck transformation of $\MTS^\grad\to\PTS$,
or equivalently, by rotating $\grad\colon\surf\to\mathbb{P}^1$ by $\pi$ clockwise.

For a curve $c:[0,1]\to \surf$, we always assume $c(t)\in\surf^\circ$ for any $t\in(0,1)$.
A \emph{graded curve} $\widetilde{c}$ is a lift of the tangent $\dot{c}$ of $c$ in $\MTS$,
of an usual curve $c$ on $\surf$.
There are exactly $\ZZ$ lifts of $\dot{c}$ on $\MTS^\grad$,
related by the grading shift $[1]$ (i.e. the deck transformation of $\MTS^\grad$).
This definition of graded curves is taken from \cite{IQZ}, see \cite{HKK} for alternative/original version.
For any graded curves $\widetilde{c}_1$ and $\widetilde{c}_2$,
let $p=c_1(t_1)=c_2(t_2)\in\surf\setminus(\partial\surf\cup\Tri)$ be an intersection of $c_1$ and $c_2$.
Note that we always require that any curves intersect transversally.
The \emph{intersection index} $i=i_p(\widetilde{c}_1,\widetilde{c}_2)$ from $\widetilde{c}_1$ to $\widetilde{c}_2$ at $p$
is the shift $[i]$ such that the lift $\widetilde{c}_2[i]\mid_p$ of $\widetilde{c}_2[i]$ at $p$ is in the interval
$$( \widetilde{c}_1\mid_p , \widetilde{c}_1[1]\mid_p )\subset \mathbb{R}T_p\surf \cong \R .$$
Note that when the index 0 intersection $p$ from $\widetilde{c}_1$ to $\widetilde{c}_2$
can be viewed as a sharp angle from the tangent direction of $\widetilde{c}_1$ at $p$ to
the one of $\widetilde{c}_2$, where sharp means that it is less than $\pi$ in $\mathbb{P}T_p\surf$.
Further details see \cite[\S~2.4]{IQZ} and cf. Figure~\ref{fig:angle}.

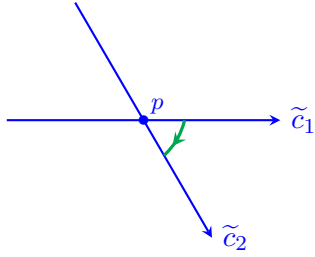
\begin{figure}[h]\centering
\begin{tikzpicture}[scale=.6]
    \foreach \j in {3,2}{
    \draw[blue,thick] (60*\j:3)\jiantou(60*\j+180:3);}
    \draw[blue,thick](0:3)node[right]{$\widetilde{c}_1$}(-60:3)node[right]{$\widetilde{c}_2$}
        (0.3,0.3)node{\footnotesize{$p$}} (0,0)node{\footnotesize{$\bullet$}};
    \draw[Green,very thick,->-=.7,>=stealth]
        (0:.9)to[bend left=15](-60:.9);
\end{tikzpicture}
\caption{Intersection index (as angle)}\label{fig:angle}
\end{figure}
	
Let $\Int^\rho(\widetilde{c}_1,\widetilde{c}_2)$ be the number of intersections
between $\widetilde{c}_1$ and $\widetilde{c}_1$ with index $\rho$ in $\surf$.
Denote by
\[\qInt(\widetilde{c}_1,\widetilde{c}_2)=\sum_{\rho\in\mathbb{Z}}
    \qv^\rho\cdot\Int^\rho((\widetilde{c}_1,\widetilde{c}_2)\]
the number of $q$-intersections between $\widetilde{c}_1$ and $\widetilde{c}_1$.
Note that $\qInt$ becomes the usual geometric intersection number when specializing $\qv=1$.

Let $\DS$ be the topological Fukaya category associated to $\gms$
and
$$\dim_\qv\Hom^\bullet(X,Y)=\sum_{d\in\ZZ} \qv^d\cdot\Hom^d(X,Y).$$
Recall the following result about $\DS$, where the first part is due to \cite{HKK}
and the second part (on morphisms) is due to \cite{IQZ}.

\begin{theorem}\cite{HKK,IQZ}\label{thm:IQZ}
There is a bijection $X$ between
the set of isotopy classes of graded curves $\{\widetilde{\eta}\}$ on $\gms$ with local system
and the set of isomorphism classes of indecomposable objects $\{X_{\widetilde \eta}\}$ in $\DS$.
Furthermore, let $\widetilde{\alpha},\widetilde{\beta}$ be two graded curves which are not closed curves (and hence no local system is needed).
Then each index $\rho$ intersection between them induces a (non-trivial) morphism in
$\Hom^\rho(X_{\widetilde \alpha}, X_{\widetilde \beta})$.
Moreover, these morphisms form a basis for the $\Hom^\bullet$ space so that we have
\begin{gather}\label{eq:int}
    \dim_\qv\Hom^\bullet(X_{\widetilde \alpha}, X_{\widetilde \beta})=\qInt(\widetilde{\alpha},\widetilde{\beta}).
\end{gather}
\end{theorem}

\subsection{Quadratic differentials}
\def\Sing{\operatorname{Sing}}
In this section, we quickly review the theory of
stability conditions as quadratic differentials in the topological Fukaya category setting.

Let $\xx$ be a compact Riemann surface and $\xi$ a non-zero \emph{meromorphic quadratic differential} on $\xx$,
that is, a meromorphic section of the square of the cotangent bundle.
The set of \emph{singularities} of $\xx$ is denoted by $\Sing(\xi)$.
Usually, the singularities considered are zeroes or poles of order $k\ge1$,
i.e. local coordinate can be chosen to be
\[
    z^{\pm k}\diff^{\otimes2}.
\]
The \emph{(horizonal) foliation} $\grad(\xi)$ of $\xi$ gives a line field (see Section~\ref{sec:GMS}) on $\xx$.
In fact, these are certain geodesics on $\xx$, where the metric is induced from $\xi$.

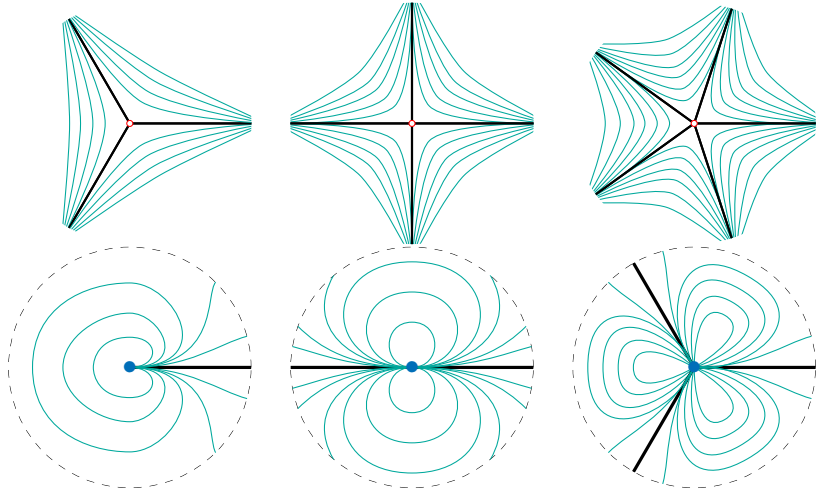
\begin{figure}[ht]\centering
\begin{tikzpicture}[scale=.4] \clip(0,0) circle (4);
\foreach \k in {1,2,0}
{    \path (120*\k:4.5) coordinate (v2)
          (120*\k+120:4.5) coordinate (v1)
          (120*\k+60:2.25) coordinate (v3)
          (0,0) coordinate (v4);
  \draw[thick](v2)to(v4)to(v1);
  \foreach \j in {.36,.54,.72,.88}
    {
      \path (v4)--(v3) coordinate[pos=\j] (m0);
      \draw[Emerald] plot [smooth,tension=.5] coordinates {(v1)(m0)(v2)};
    }
}
\draw[red,fill=white](0,0)circle(.1);
\end{tikzpicture}
\quad
\begin{tikzpicture}[scale=.4] \clip(0,0) circle (4);
\foreach \k in {0,1,2,3}
{    \path (90*\k:4.5) coordinate (v2)
          (90*\k+90:4.5) coordinate (v1)
          (90*\k+45:2.25) coordinate (v3)
          (0,0) coordinate (v4);
  \foreach \j in {.36,.54,.72,.88,1.08}
    {
      \path (v4)--(v3) coordinate[pos=\j] (m0);
      \draw[Emerald] plot [smooth,tension=.5] coordinates {(v1)(m0)(v2)};
    }
  \draw[thick](v2)to(v4)to(v1);
}
\draw[red,fill=white](0,0)circle(.1);
\end{tikzpicture}
\quad
\begin{tikzpicture}[scale=.4] \clip(0,0) circle (4);
\foreach \k in {0,1,2,3,4}
{    \path (72*\k:4.5) coordinate (v2)
          (72*\k+72:4.5) coordinate (v1)
          (72*\k+36:2.25) coordinate (v3)
          (0,0) coordinate (v4);
  \foreach \j in {.36,.54,.72,.88,1.08,1.3}
    {
      \path (v4)--(v3) coordinate[pos=\j] (m0);
      \draw[Emerald] plot [smooth,tension=.5] coordinates {(v1)(m0)(v2)};
    }
  \draw[thick](v2)to(v4)to(v1);
}
\draw[red,fill=white](0,0)circle(.1);
\end{tikzpicture}

\begin{tikzpicture}[scale=.4,rotate=0]\clip(0,0) circle (4);
\draw[dashed, thin](0,0)circle(4);
\draw[very thick](0,0)to(0:4)(-90:4.5);
\draw[Emerald] (0,0)
    .. controls +(0:2) and +(195:1) ..(15:4);
\draw[Emerald] (0,0)
    .. controls +(0:2) and +(165:1) ..(-15:4);
\draw[Emerald] (0,0)
    .. controls +(0:3) and +(225:.5) ..(45:4);
\draw[Emerald] (0,0)
    .. controls +(0:3) and +(125:.5) ..(-45:4);

\draw[Emerald] (0,0)
    .. controls +(0:3) and +(0:1.5) ..(90:2.8)
    .. controls +(180:1.5) and +(90:2) ..(180:3.2)
    .. controls +(-90:2) and +(180:1.5) ..(-90:2.8)
    .. controls +(0:1.5) and +(0:3) .. (0,0);
\draw[Emerald] (0,0)
    .. controls +(0:2) and +(0:1) ..(90:1.8)
    .. controls +(180:1) and +(90:1) ..(180:2.2)
    .. controls +(-90:1) and +(180:1) ..(-90:1.8)
    .. controls +(0:1) and +(0:2) .. (0,0);
\draw[Emerald] (0,0)
    .. controls +(0:1) and +(0:1) ..(90:1)
    .. controls +(180:1) and +(90:.3) ..(180:1.2)
    .. controls +(-90:.3) and +(180:1) ..(-90:1)
    .. controls +(0:1) and +(0:1) .. (0,0);
\draw[NavyBlue](0,0)node{$\bullet$};
\end{tikzpicture}
\quad
\begin{tikzpicture}[scale=.4,rotate=0]\clip(0,0) circle (4);
\draw[dashed, thin](0,0)circle(4);
\draw[very thick](180:4)to(0:4)(-90:4.5);
\draw[Emerald] (0,0)
    .. controls +(0:2) and +(195+30:1) ..(20:4);
\draw[Emerald] (0,0)
    .. controls +(0:2) and +(165-30:1) ..(-20:4);

\draw[Emerald] (0,0)
    .. controls +(0:2) and +(195+105:2) ..(50:4);
\draw[Emerald] (0,0)
    .. controls +(0:2) and +(165-105:2) ..(-50:4);

\draw[Emerald] (0,0)
    .. controls +(180:2) and +(180-195-105:2) ..(180-50:4);
\draw[Emerald] (0,0)
    .. controls +(180:2) and +(180-165+105:2) ..(180--50:4);

\draw[Emerald] (0,0)
    .. controls +(0:2) and +(195:1) ..(10:4);
\draw[Emerald] (0,0)
    .. controls +(0:2) and +(165:1) ..(-10:4);

\draw[Emerald] (0,0)
    .. controls +(180:2) and +(15:1) ..(195-5:4);
\draw[Emerald] (0,0)
    .. controls +(180:2) and +(-15:1) ..(165+5:4);

\draw[Emerald] (0,0)
    .. controls +(180:2) and +(45:1) ..(200:4);
\draw[Emerald] (0,0)
    .. controls +(180:2) and +(-45:1) ..(160:4);

\draw[Emerald] (0,0)
    .. controls +(0:3) and +(0:3) ..(90:3.5)
    .. controls +(180:3) and +(180:3) .. (0,0);
\draw[Emerald] (0,0)
    .. controls +(0:2) and +(0:2) ..(90:2.5)
    .. controls +(180:2) and +(180:2) .. (0,0);
\draw[Emerald] (0,0)
    .. controls +(0:1) and +(0:1) ..(90:1.5)
    .. controls +(180:1) and +(180:1) .. (0,0);
\draw[Emerald] (0,0)
    .. controls +(0:3) and +(0:3) ..(-90:3.5)
    .. controls +(180:3) and +(180:3) .. (0,0);
\draw[Emerald] (0,0)
    .. controls +(0:2) and +(0:2) ..(-90:2.5)
    .. controls +(180:2) and +(180:2) .. (0,0);
\draw[Emerald] (0,0)
    .. controls +(0:1) and +(0:1) ..(-90:1.5)
    .. controls +(180:1) and +(180:1) .. (0,0);

\draw[NavyBlue](0,0)node{$\bullet$};
\end{tikzpicture}
\quad
\begin{tikzpicture}[scale=.4,rotate=0]\clip(0,0) circle (4);
\draw[dashed, thin](0,0)circle(4);
\foreach \j in {0,1,2}
{
  \draw[very thick](120*\j+0:4)to(0,0);
\draw[Emerald] (0,0)
    .. controls +(120*\j+0:2) and +(120*\j+195:1) ..(120*\j+15:4);
\draw[Emerald] (0,0)
    .. controls +(120*\j+0:2) and +(120*\j+165:1) ..(120*\j+-15:4);

\draw[Emerald] (0,0)
    .. controls +(0+120*\j-10:1.5) and +(-30+120*\j:3) ..(60+120*\j:3.5)
    .. controls +(150+120*\j:3) and +(120+120*\j+10:1.5) .. (0,0);

\draw[Emerald] (0,0)
    .. controls +(0+120*\j:1) and +(-30+120*\j:2.5) ..(60+120*\j:3)
    .. controls +(150+120*\j:2.5) and +(120+120*\j:1) .. (0,0);

\draw[Emerald] (0,0)
    .. controls +(0+120*\j+15:.6) and +(-30+120*\j:1.8) ..(60+120*\j:2.5)
    .. controls +(150+120*\j:1.8) and +(120+120*\j-15:.6) .. (0,0);

\draw[Emerald] (0,0)
    .. controls +(0+120*\j+30:.6) and +(-30+120*\j:1) ..(60+120*\j:2)
    .. controls +(150+120*\j:1) and +(120+120*\j-30:.6) .. (0,0);
}
\draw[NavyBlue](0,0)node{$\bullet$};
\end{tikzpicture}
\caption{Foliations/line fields near zeroes/poles}\label{fig:foli}
\end{figure}
For instance,
near a zero of order 1/2/3, the foliation $\grad(\xi)$ is shown in the upper pictures of Figure~\ref{fig:foli};
near a pole of order 3/4/5, the foliation $\grad(\xi)$ is shown in the lower pictures of Figure~\ref{fig:foli}.
When performing \emph{real blow-up} at a (higher order) pole $p$ of order $k\ge3$,
one gets a boundary component $\partial_p$ with $k-2$ marked points,
where points on $\partial_p$ correspond to tangent direction at $p$
and marked points are distinguished directions, as shown (the black lines) in Figure~\ref{fig:foli}.
In the pictures we presented, our convention is the following:
red circles are zeros and blue bullets are poles.
For details, see \cite{BS,KQ1}.

\begin{figure}[hb]\centering
\includegraphics[height=2in,width=2.5in]{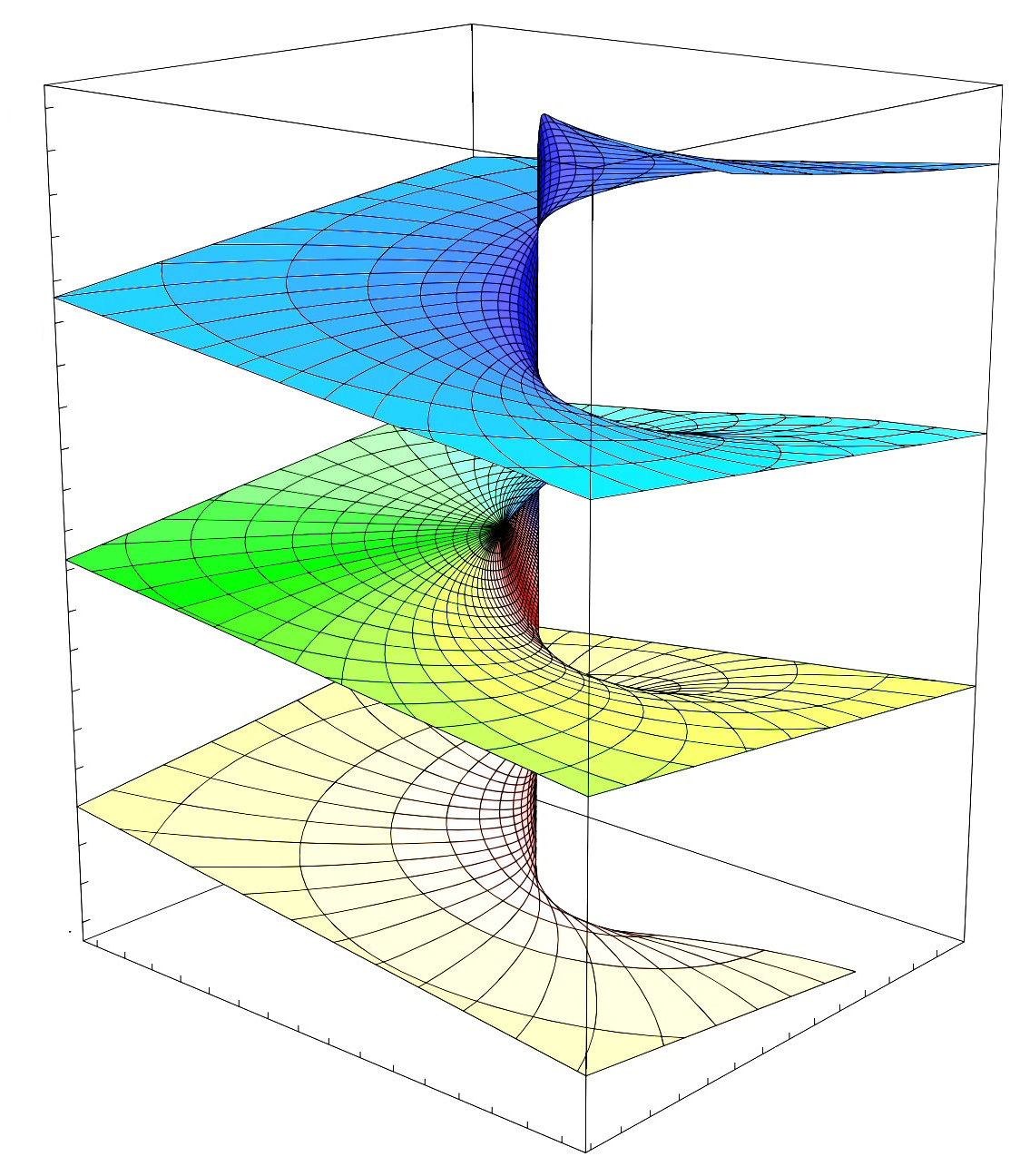}
\quad
\begin{tikzpicture}[scale=.36,arrow/.style={->,>=stealth,thick}]
\newcommand{\vtex}{{$\bullet$}}
\draw
    (180:7)coordinate (Z1)
    (0:7)coordinate (Z3)
    (-90:7)coordinate (Z2)

    (60:7)coordinate (U2)
    (120:7)coordinate (U1)
    (-135:7)coordinate (U3)
    (-45:7)coordinate (U4);

  \foreach \j in {.1,.2,...,.9}
    {
      \path (Z1)--(Z2) coordinate[pos=\j] (m0);
      \path (U1)--(m0) coordinate[pos=.9] (m1);
      \draw[Emerald, very thin] plot [smooth,tension=.5] coordinates {(U1)(m1)(U3)};
    }
  \foreach \j in {.1,.2,...,.9}
    {
      \path (Z3)--(Z2) coordinate[pos=\j] (m0);
      \path (U2)--(m0) coordinate[pos=.9] (m1);
      \draw[Emerald, very thin] plot [smooth,tension=.5] coordinates {(U2)(m1)(U4)};
    }

  \foreach \j in {-25,-16,-9,0,9,16,25,36,49,64,81,100,121,144,169,196,225,256,289,324}
    {
      \path (U1)--(U2) coordinate[pos=.5] (m0);
      \path ($(m0)!\j*.0025!(Z2)$) coordinate (m1);
      \draw[Emerald, very thin] plot [smooth,tension=.5] coordinates {(U2)(m1)(U1)};
    }

\draw[cyan, thick](Z1)to(U1)to(Z2)to(U3)--cycle;
\draw[cyan, thick](Z3)to(U2)to(Z2)to(U4)--cycle;
\foreach \j in {10,20,30,40,50,-25,-35}
{\draw[cyan,thick] (Z1)to(180-\j:7);}
\foreach \j in {10,20,30,40,50,-25,-35}
{\draw[cyan,thick] (Z3)to(\j:7);}
\foreach \j in {-25,25,35,-35}
{\draw[cyan,thick] (Z2)to(-90+\j:7);}
\draw[thick] (0,0) circle (7);

\draw[red,thick](Z1)to[bend left=10](Z2)to[bend left=10](Z3);
\foreach \j in {1,2,3}{
    \draw(Z\j)node[white]{$\bullet$} node[red]{$\circ$};}
\end{tikzpicture}
\caption{The Riemann surface of $\log z$ and foliation of $A_2$ type graded marked surface}\label{fig:RS}
\end{figure}
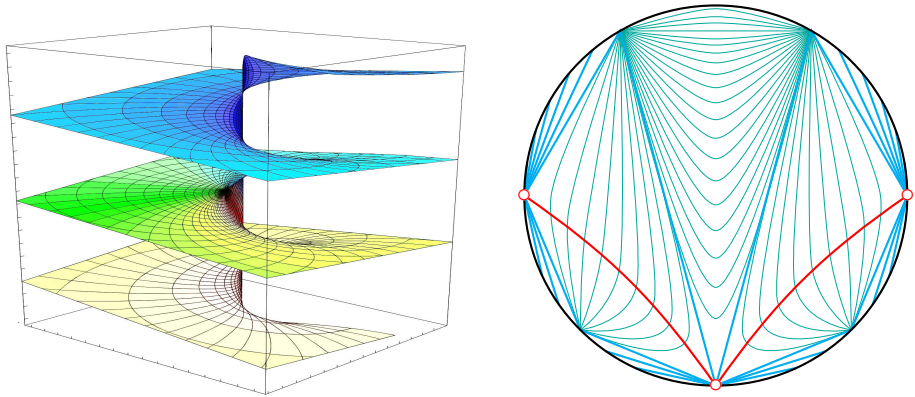

However, in our case, the singularities are of \emph{exponential type}, in the sense of \cite{HKK},
cf. \cite{IQ2}.
Namely, the local coordinate around a given singularity $p$ is of the form (up to scaling a holomorphic function)
\begin{equation}\label{eq:k+l}
    z^{-l} e^{ z^{-k} }\diff z^{\otimes2},
\end{equation}
where the numerical data here is $(k,l)\in\ZZ_{>0}\times\ZZ$
and $l$ can be calculated as two minus the winding number of the line field around $p$ (cf. \cite{IQ2}).
When performing real blow-up at a (higher order) pole $p$ of type $(k,l)$,
one gets a boundary component $\partial_p$ with $2k$ distinguished points:
\begin{itemize}
\item $k$ of which are in the metric completions (call \emph{closed marked points}) that behave as infinity order zeroes;
\item $k$ of which (called \emph{open marked points}) behave as infinity order poles.
\end{itemize}
The closed and open marked points are in alternative order on $\partial_p$.

The neighbourhood of such an infinity order zero/pole is as
the neighbourhood of zero in the Riemann surface for $\log z$,
cf. the left picture (taken from \cite{Wiki}) of Figure~\ref{fig:RS}
(and thus they should sometimes be considered as marked/unmarked boundary arcs, cf. \cite{HKK}).
Also, the foliation $\grad(\xi)$ on a real blow-up of a Riemann sphere with a single singularity of type $(3,4)$
is shown in the right picture (taken from \cite{IQ2}) of Figure~\ref{fig:RS}.
Our convention is that
red circles are closed marked points and we do not use points to represent open ones in this paper.

Denote by $\xxc^\xi$ the graded marked surface (of closed type),
which is the real blow-up of $\xx$ with respect to $\xi$ equipped with closed marked points $\Y(\xi)$
and foliation $\grad(\xi)$ as its grading.

The foliations induce the \emph{horizontal strip decomposition} of $\surf$ (cf. \cite[\S~2.4]{HKK}),
where the surface is divided into regions/strips consisting of horizontal foliations.
Each strip is either isomorphic to the upper half plane $\mathbf{H}$ (with finite height)
or a strip $\{z\in\CC\mid 0\le \Im (z)\le \Im (z_0)\}$ for some $z_0\in\CC$ with $\Im (z_0)>0$
(with finite height).

An infinite height strip is shown in right picture of Figure~\ref{fig:h.s.}.
A finite height strip is shown in the right picture of Figure~\ref{fig:h.s.}.
In this case,
there is exactly one closed marked point on each boundary of a strip mentioned above,
namely $0$ and $z_0$ respectively.
So there is an unique geodesic connecting these two points,
which is known as the \emph{saddle connection} of this strip, whose angle is $\arg z_0$.
In fact, up to a small rotation of the quadratic differential, we can assume that
there is no horizonal saddle trajectories (known as \emph{saddle-free}),
so that any finite height strip is as the case mentioned above.

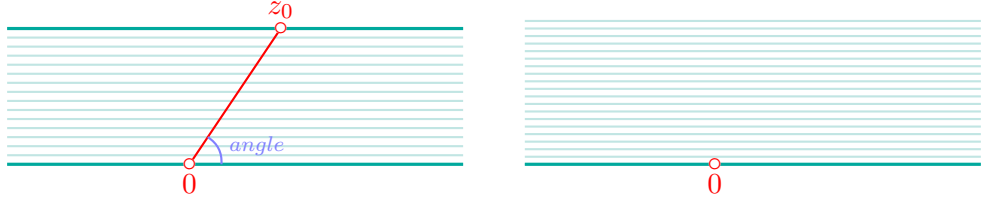
\begin{figure}[h]\centering
\begin{tikzpicture}[yscale=.12,xscale=.12]
\foreach \j in {0,...,15}
    {\draw[Emerald!23, thick](-20,\j)to(30,\j);}
\foreach \j in {0,15}
    {\draw[Emerald, very thick](-20,\j)to(30,\j);}
\draw[red,thick](0,0)to(10,15);
\draw[white](0,0)\nn(10,15)\nn;
\draw[red](0,0)node[below]{$0$}\ww(10,15)\ww node[above]{$z_0$};
\draw[blue!50,thick](3.5,0)to[bend left=-30]node[right]{$^{angle}$}(2,3);
\end{tikzpicture}\qquad
\begin{tikzpicture}[yscale=.1,xscale=.1]
\foreach \j in {0,...,19}
    {\draw[Emerald!23, thick](-25,\j)to(35,\j);}
\foreach \j in {0}
    {\draw[Emerald, very thick](-25,\j)to(35,\j);}
\draw[white](0,0)\nn;
\draw[red](0,0)\ww node[below]{$0$};
\end{tikzpicture}
\caption{Horizontal strips: finite and infinite height types}\label{fig:h.s.}
\end{figure}
In general, a saddle connection is a maximal geodesic connecting zeroes (or points in the metric completion).
Thus, one can integral the (square root of) the quadratic differential along saddle connections.
\subsection{Stability conditions as quadratic differentials}

\begin{definition}
A \emph{$\gms$-framed quadratic differential} $\Xi=(\xx,\xi,\psi)$
consists of a Riemann surface $\xx$, a meromorphic quadratic differential $\xi$ with only exponential type singularities
and a diffeomorphism $\psi\colon\gms \to(\xxc^\xi,\grad(\xi))$ preserving marked points.
Two $\gms$-framed quadratic differentials $(\xx_1,\xi_1,\psi_1)$ and $(\xx_2,\xi_2,\psi_2)$
are equivalent, if there exists a biholomorphism $f\colon\xx_1\to\xx_2$
such that $f^*(\phi_2)=\phi_1$ and $\psi_2^{-1}\circ f_*\circ\psi_1$ is a homeomorphism of
$\gms$ that is isotopic to identity.
Denote by $\FQuad_\infty(\gms)$ the moduli space of $\gms$-framed quadratic differentials on $\gms$.
\end{definition}

The main result of \cite{HKK} is the following, where the surjectivity part is improved by \cite{T}.

\begin{theorem}\label{thm:HKKT}\cite{HKK,T}
There is an isomorphism of complex manifolds
\begin{equation}\begin{array}{rcl}
  \iota=\iota(\gms) \colon\FQuad_\infty(\gms)
    &\xrightarrow{\cong}&\Stab\D_\infty(\gms),
\\ \xi& \mapsto & \sigma .
\end{array}
\end{equation}
Moreover, the graded saddle connections of $\Xi$
correspond (semi)stable objects of $\sigma$
under the bijection $X$ in Theorem~\ref{thm:IQZ}
and, up to $2\pi\ZZ$, the angles of a saddle connection equals $\pi$ times the phase of the corresponding semistable object.
\end{theorem}

\subsection{Winding numbers}\label{sec:winding}
\def\uk{\mathbf{k}}
\def\ul{\mathbf{l}}
\def\uw{\mathbf{w}}

Denote by $\num(\gms)=(\uk,\uw)$ the partial numerical data of $\gms$,
for $\uk=(k_1,\ldots,k_b)$ and $\uw=(w_1,\ldots,w_b)$,
where $k_i$ is the number of closed marked points on a boundary component $\partial_i$
and $w_i$ the \emph{(clockwise) winding number} around $\partial_i$.
Note that comparing with \eqref{eq:k+l}, we have $w_i=2-l_i$ and
they satisfy
\[
    \sum_{i=1}^b l_i=4-4g\quad\Longleftrightarrow\quad \sum_{i=1}^b w_i=4g-4+2b
\]
(see \cite{LP2,IQ2} for details).

We are interested in a particular class of arcs on $\surf$, i.e. the minimal arcs.
A \emph{minimal arc} on $\surf$ is an arc connecting two adjacent closed marked points on
some boundary component, such that it is isotopic to a boundary segment.
For instance, the arcs $\eta_j$ in Figure~\ref{fig:Apq} are minimal arcs.

\begin{example}\label{ex:Apq}
Consider the case that $\surf$ is an annulus with boundaries $\partial_m$ and $\partial_r$.
Then $\DS$ is triangle equivalent to the bounded derived category $\Dpq$ of a graded $\widetilde{A_{m,r}}$ quiver
(with $m+r$ vertices whose arrows form a non-oriented cycle,
$m$ of which are clockwise and the other $r$ are anticlockwise).
\[
\begin{tikzpicture}[yscale=.5]
\draw(0.5,0) node[left]{$\widetilde{A_{m,r}}\colon\qquad$} node (x1) {} (1,2) node (x2) {} (2,3) node (x3) {}  (5,-1) node (y1){};
\draw(0.5,0) node (x1) {} (1,-2) node (x4) {} (2,-3) node (x5) {} ;
\foreach \j/\i in {1/2,2/3,1/4,4/5}{
\draw[->,>=stealth](x\j)\ww to (x\i) \ww ;}
\draw[->,>=stealth](x3) to (2.7,3) node[right,rotate=0]{$\cdots$};
\draw[->,>=stealth](x5) to (2.6,-3) node[right,rotate=-0]{$\cdots$};
\draw[->,>=stealth](3.7,3) to  (y1)\ww;
\draw[->,>=stealth](3.5,-3) to  (y1)\ww;
\end{tikzpicture}
\]
Note that the sum of the winding numbers is zero in this case.
Then norm form of the numerical data can be chosen to be
$\num(\gms)=((m,r),(w,-w))$ for $m,r\in\ZZ_+$ and $w\in\ZZ_{\ge0}$.
\begin{lemma}\label{lem:gd}
$\gd \Dpq\le 1+w/m$.
\end{lemma}
\begin{proof}
When $w=0$, we have $\gd\Dpq=1$, which was calculated in \cite{Q1}.
Now assume that $w>0$.
Let $\eta_1,\ldots,\eta_m$ be the minimal arcs on $\partial_m$ in clockwise order,
as shown (red arcs) in Figure~\ref{fig:Apq}.
By convention, the subscript will be in $\ZZ_m=\ZZ/m\ZZ$.
\begin{figure}[h]\centering
\begin{tikzpicture}[xscale=-.7,yscale=.7,arrow/.style={->,>=stealth,thick}]
\draw[thick,fill=gray!11] (0,0) circle (1);\draw[thick](0,0) circle (3)(90:1)edge[red]node[left]{$\alpha$}(90:3);
\foreach \j in {1,2,0} {
    \draw[red,thick]plot [smooth,tension=1.2] coordinates { (120*\j+90:1)(120*\j+30:1.5)(120*\j-30:1) }
        (120*\j-30:.6) node[black]{\footnotesize{$p_{\j}$}};
}
\draw[red,thick,dashed](90:1)to(90:3)(90+60:3)\ww to[bend left=5](90+120:3)
    (90:3)to[bend left=-5](90-60:3)\ww to[bend left=-5](90-120:3)to[bend left=35](90+120:3);

\foreach \j in {1,2,3} {\draw[red,thick,dashed] (120*\j+90:3)\ww (120*\j+90:1)\ww
    (120*\j-90:1.8)node{\footnotesize{$\eta_{\j}$}} ;}
\end{tikzpicture}
\caption{A full formal arc system containing certain minimal arcs in the annulus case}\label{fig:Apq}
\end{figure}
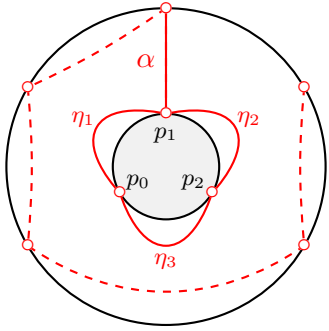

For any graded lifts $\widetilde{\eta_j}$ of $\eta_j$, we have
\begin{gather}\label{eq:index sum}
 \sum_{j=1}^m i_{p_j}( \widetilde{\eta_j} , \widetilde{\eta_{j+1}} )=w+m,
\end{gather}
where $i_{p_j}$ is the intersection index of $\widetilde{\eta_j}$ and $\widetilde{\eta_{j+1}}$ at $p_j$,
cf. Figure~\ref{fig:Apq}.
Therefore, we can choose certain graded lifts of $\eta_j$ such that
\[
    i_{p_j}( \widetilde{\eta_j} , \widetilde{\eta_{j+1}} )=
    \lfloor j(m+w)/m \rfloor
    -\lfloor (j-1)(m+w)/m \rfloor
    ,\quad \forall j.
\]
In particular, $i_{p_1}( \widetilde{\eta_1} , \widetilde{\eta_{2}} )\ge2$ and
\[
    \lfloor 1+w/m \rfloor  \le
    i_{p_j}( \widetilde{\eta_j} , \widetilde{\eta_{j+1}} )\le
        \lfloor 1+w/m \rfloor +1.
\]
Then we can complete $\{ \widetilde{\eta_j} \mid j\in\ZZ_m \}$ to a full formal arc system $\ac$
(cf. dashed arcs in Figure~\ref{fig:Apq}) such that
\begin{itemize}
\item there is exactly one arc $\alpha$ that is incident $\partial_m$ at $p_1$ and connects two boundaries;
\item for any two graded arcs $\eta,\eta'$ in $\ac-\{ \widetilde{\eta_j} \mid j\in\ZZ_m \}$,
there is at most one intersection between them and,
if they intersect, the intersection index is 1.
\end{itemize}
Here, a full formal arc system is a collection of (graded) arcs that divide $\surf$ into polygons,
such that each polygon contains exactly one boundary segment.
The objects corresponding to a full formal arc system is a set of generators for $\Dpq$.
The condition on intersection index can be translated to
\begin{gather}
    \Hom^{\le0}(X_{\widetilde \eta}, X_{\widetilde \eta'} )=0,\quad \forall\eta,\eta'\in\ac.
\end{gather}
Thus $\{ X_{\widetilde \eta} \mid \eta\in\ac \}$ form a so-called \emph{simple minded collection}.
Equivalently, they are the set of simple objects of a heart $\h_\ac$ that they generate $\Dpq$.
Furthermore (cf. \cite[Lem~5.2]{B1}), to give a stability condition $\sigma$ in $\Stab\D_\infty(\gms)$
with heart $\h_\ac$ is equivalent to make a choice of central charges for simples
$$\{ Z(X_{\widetilde \eta})\in\mathbf{H} \mid \eta\in\ac \},$$
where $$\mathbf{H}=\{ z=r e^{\bi \pi \theta}\mid r\in\RR_{>0}, \theta\in [0,1) \}\subset\CC$$
is the upper half plane.
We can find a stability condition $\sigma$ such that
\begin{gather}
\begin{cases}
   \phi_\sigma(  X_{\widetilde \eta_{j+1}}  )= j(m+w)/m - \lfloor j(m+w)/m\rfloor,
    & \forall j\\
   \phi_\sigma(  X_{\widetilde \eta'}  )=\phi_\sigma(  X_{\widetilde \eta}  ),
    & \text{for any $\eta,\eta'\notin\ac-\{\eta_j\mid j\in\ZZ_m\}$,}\\
   \phi_\sigma(  X_{\widetilde \alpha}  )=\phi_\sigma(  X_{\widetilde \eta_1}  ).
\end{cases}
\end{gather}
Then we have
\[
    \phi_\sigma(  X_{\widetilde \eta_{j+1}}  )-\phi_\sigma(  X_{\widetilde \eta_j}  )=1+w/m,
\]
for any $j$.
One can check that $\gldim\sigma=1+w/m$
(or use Proposition~\ref{pp:reach1+}),
which completes the lemma.
\end{proof}
\end{example}
\section{Contractible flow}
\subsection{General strategy}
In this section, we develop a strategy to attack the contractibility conjecture
of spaces of stability conditions.
The idea is to use the function $\gldim$ to induce a contractible flow.
Of cause, this strategy should only apply to the `Calabi-Yau-$\infty$' case,
as $\gldim$ is constant on Calabi-Yau-$N$ categories (for $N\in\ZZ_+$).

\begin{definition}\label{def:CP}
Given a stability condition $\sigma$, define a set $\CP{\sigma}$
\begin{gather}\label{eq:CP}
\{ (M_1,M_2) \mid M_i\in\Sim\hh{P}(\phi_i), M_1[\ZZ]\neq M_2[\ZZ], \Hom(M_1,M_2)\neq0, \phi_2-\phi_1=\gldim\sigma  \}
\end{gather}
which consists of pairs of stable indecomposable objects whose phase difference achieves
the value $\gldim\sigma$ and has non-zero morphisms in $\hh{P}$.
Note that we also require that the objects in such a pair are not in the same shift orbit
to exclude the case of nontrivial higher self-extension of an object.

Define a subspace
\begin{equation}\label{eq:subspace}
    \Stab_{\CP{\sigma}}\D\colon= \{\sigma'\in\Stab\D\mid \CP{\sigma'}=\CP{\sigma} \}.
\end{equation}
\end{definition}

We will prove that \eqref{eq:subspace} is determined by a collection of equations and the following conjecture,
which allow us to use the differential of $\gldim$ to contract (certain part of) the space of stability conditions piecewise.

\begin{conjecture}\label{cond}
$\Stab_{\CP{\sigma}}\D$ is a real submanifold of $\Stab\D$
where the function $\gldim$ is differentiable without critical points restricted to its interior.
\end{conjecture}

Moreover, we expect the following for many cases,
which holds for the case of coherent sheaves on the projective plane (cf. \cite{Q1,FLLQ}).
\begin{conjecture}\label{conj:contract}
The differential of $\gldim$ provide a flow such that
$\Stab_{<x}\D$ contracts to $\Stab_{<y}\D$ for any $\gd\D<y<x$.
If in addition that $\D$ is $\gldim$-reachable, then
$\Stab_{<x}\D$ contracts to $\Stab_{y}\D$ for $y=\gd\D$ and any $y<x$.
\end{conjecture}

\subsection{Max angle as gldim and reachability}
Recall that the rank $n$ of $\surf$ in \eqref{eq:rank} is required to be at least 2.
We apply the general strategy above to the topological Fukaya categories $\DS$.
\def\core{\operatorname{Core}}

Take $\sigma\in\Stab\DS$ with $\CP{\sigma}$ as in Definition~\ref{def:CP}.
Let $\Xi=(\xx,\xi,\psi)=\iota^{-1}(\sigma)$ be the $\gms$-framed quadratic differential as in Theorem~\ref{thm:HKKT}.
We will identify $(\surf, \Y, \grad)$ with $(\xx^\xi,\Y(\xi),\lambda(\xi))$ via $\psi$
when there is no confusion.

\begin{lemma}\label{lem:bb}
Let $M_1,M_2$ be two $\sigma$-semistable indecomposable objects with
corresponding graded curves $\widetilde{\gamma_i}$ on $\gms$, respectively.
If $\widetilde{\gamma_1}$ and $\widetilde{\gamma_2}$ intersect in the interior of $\surf$ of index 0, then
\begin{gather}\label{eq:ineq}
    \phi_\sigma(M_1) < \phi_\sigma(M_2) < \phi_\sigma(M_1)+1
\end{gather}
\end{lemma}
\begin{proof}
Since $\widetilde{\gamma_1}$ intersects $\widetilde{\gamma_2}$ at a point $p$ with index $0$,
$\widetilde{\gamma_2}$ intersects $\widetilde{\gamma_1}[1]$ at $p$ with index $0$.
By \eqref{eq:int} in Theorem~\ref{thm:IQZ}, we have
\[
    \Hom(M_1,M_2)\ne0\ne\Hom(M_2,M_1[1]).
\]
As $M_1$ and $M_2$ are both $\sigma$-semistable, we have
$\phi_\sigma(M_1) \leq  \phi_\sigma(M_2) \leq \phi_\sigma(M_1)+1$.
To get \eqref{eq:ineq}.
we need to rule out the possible equality.
Suppose that $\phi_\sigma(M_1[k]) =  \phi_\sigma(M_2)$ for $k\in\{0,1\}$.
After rotating an angle of $-\phi_\sigma(M_2)\cdot \pi$,
the underlying curves $\gamma_i$ become horizontal foliations.
But such foliations can not intersect in the interior of $\surf$, which is a contradiction.
\end{proof}

An immediate consequence is the following.

\begin{corollary}\label{cor:bb}
Let $M_1,M_2$ be two $\sigma$-semistable indecomposable objects with
corresponding graded curves $\widetilde{\gamma_i}$ on $\gms$, repsectively.
If $\phi_\sigma(M_2)-\phi_\sigma(M_1)\ge1$ with $\Hom(M_1,M_2)\neq0$,
then
\begin{itemize}
  \item either $\gamma_1$ and $\gamma_2$ intersect (and only intersect) at marked points in $\Y$,
  \item or $\gamma_1=\gamma_2$ that corresponds the same simple closed curve.
\end{itemize}
In the latter case,
$M_2$ is some shift of $M_1$ and
$\gamma_i$ corresponds to a ring domain in the foliation of $\Xi$ with angle $\phi_\sigma(M_i)\cdot \pi$.
\end{corollary}
\begin{proof}
The lemma above shows that there is no intersection between $\gamma_i$ in the interior of $\surf$.
If $\gamma_1\neq\gamma_2$ and they do not intersect,
then Theorem~\ref{thm:IQZ} implies that there is no $\Hom$ between $M_i$, which is a contradiction.
Thus, the only cases left is the ones listed in the corollary.
Note that in the latter case, $\gamma_i$ can not have self intersection since it is a foliation of a fixed angle.
\end{proof}

\begin{definition}
Consider the set $\T$ of all saddle connections of $\Xi$,
which corresponds to the set $\Ind\hh P$ of all $\sigma$-semistable indecomposable objects,
where $\sigma=(Z,\hh{P})$ is the stability condition that corresponds to $\Xi$.
Denote by $\core(\xi)$ the \emph{core} of $\xi$, which is the convex hull of $\T$.
\end{definition}
At each marked point $p\in\Y(\Xi)$, denote by $( \eta_-^p,\ldots,\eta_+^p )$
the set of all \emph{ungraded} saddle connections in clockwise order (with respect to $p$).
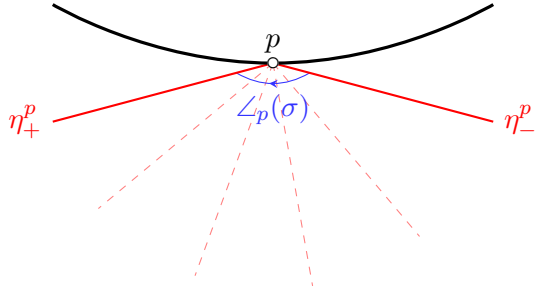
\begin{figure}[ht]\centering
\begin{tikzpicture}[xscale=-1,yscale=1]
\foreach \j in {-1,-4,2,5}
    {\draw[red!50,thin,dashed] (0,0) edge (-90+10*\j:3);}
\draw[->-=.55,>=stealth,blue!80](-165:.5)to[bend right]node[below]{$\angle_p(\sigma)$}(-15:.5);
\draw[thick,red](-15:3)node[left]{$\eta_+^p$}to(0,0)to(-165:3)node[right]{$\eta_-^p$};
\draw[very thick] plot [smooth,tension=1] coordinates {(165:3)(0,0)(15:3)};
\draw (0,0)\ww node[above]{$p$};
\end{tikzpicture}
\caption{Max angle}\label{fig:max}
\end{figure}

\begin{proposition}\label{cor:bb2}
The core $\core(\xi)$ and all $\eta_\pm^p$ are well defined.
\end{proposition}
\begin{proof}
By \cite[Prop.~2.2]{HKK}, $\core(\xi)$ is the union of finite saddle connections and triangles cut out by $\mathbf{A}$,
for any maximal geodesic arc system $\mathbf{A}$.
Moreover, it is well-defined and independent of the choice of $\mathbf{A}$.
Consider the boundaries of these triangles together with the saddle connections (for any chosen $\mathbf{A}$),
we see that $\eta_\pm^p$ must be among them.
\end{proof}

\begin{remark}
It is possible that there are infinite many saddle trajectories coming out of a marked point $p\in\Y(\Xi)$,
e.g. in the annulus case of \S~\ref{sec:Apq}, cf. Figure~\ref{fig:wind around}.
However, there are still leftmost/rightmost saddle trajectories bounding all of them.
For instance, the orange loop in Figure~\ref{fig:wind around}, which corresponds to (some shifts of) a skyscrpter sheaf in the usual Kronecker case. See Example~\ref{ex:K2} for more details in the Kronecker case.
\begin{figure}[h]\centering
\begin{tikzpicture}[scale=.9,arrow/.style={->,>=stealth,thick}]
\begin{scope}[shift={(0,0)}]
\draw[very thick,fill=gray!11] (0,0) circle (.5);\draw[ultra thick](0,0) circle (3);
\draw[red, very thin]
    (0,3).. controls +(230:2) and +(180:1.5) ..(0,-.5);
\draw[red, thin]
    (0,3).. controls +(220:3) and +(180:2) ..(0,-1.4)
    arc(-90:90:1.2)(0, 1)arc(90:270:.75)(0,-.5);
\draw[red,thick,dashed]
    (0,3).. controls +(210:3) and +(180:3) ..(0,-2)arc(-90:90:1.8)(0,1.6)arc(90:270:1.3)(0,-1);
\draw[orange, ultra thick]
    (0,3).. controls +(200:4) and +(180:3) ..(0,-2.75)
    (0,3).. controls +(-20:4) and +(0:3) ..(0,-2.75);
\foreach \j in {0} {\draw[red,thick,dashed] (120*\j+90:3)\ww (120*\j-90:.5)\ww ;}
\end{scope}
\end{tikzpicture}
\caption{Infinite saddles at a closed marked point}\label{fig:wind around}
\end{figure}
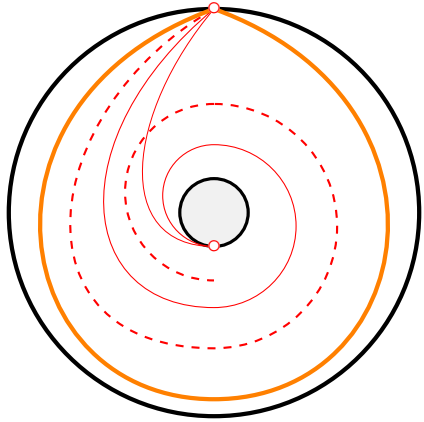
\end{remark}

Choose a grading $\widetilde{\eta_\pm^p}$ for both of them so that the intersection indexes are zero, i.e.
$i_p(\widetilde{\eta_-^p},\widetilde{\eta_+^p})=0$.
Note that $\eta_\pm^p$ may be the two endpoints of the same arc,
and in such a case, their graded version $\widetilde{\eta_\pm^p}$ may still differ by shifts.
Let $M_\pm^p$ be the $\sigma$-semistable object corresponding to $\widetilde{\eta_\pm^p}$
with proper shifts, such that the intersection of $\eta_\pm^p$ at $p$ induces a non-zero homomorphism in $\Hom(M_-^p,M_+^p)$.
Denote by
\[
    \angle_p\core(\sigma)=\phi_\sigma(M_+^p)-\phi_\sigma(M_-^p).
\]
Note that $M_\pm^p$ are only well-defined up to some shifts simultaneously
but $\angle_p\core(\sigma)$ is independent of such shifts.

Now we can describe a formula for $\gldim$ under certain conditions.

\begin{proposition}\label{pp:reach1}
If $\gldim\sigma>1$, then $\sigma$ is $\gldim$-reachable and
\begin{gather}\label{eq:reach1}
  \gldim\sigma=\max\angle\core(\xi)\colon=\max_{p\in\Y}\{ \angle_p\core(\sigma) \}.
\end{gather}
Moreover, if $\gldim\sigma\notin\ZZ$,
then any pair $(M_1,M_2)$ in \eqref{eq:CP} corresponds to an angle of $\core(\xi)$.
\end{proposition}
\begin{proof}
Firstly, consider the case when $\gldim>1$.
Let $x=\gldim\sigma$.
For any $0<\epsilon\ll1$ such that
\begin{gather}\label{eq:noZZ}
\big(x-\epsilon,x\big)\cap \ZZ=\emptyset,
\end{gather}
take any $y\in\big(x-\epsilon,x\big)$ which is achieved by
\[
    y=\phi_\sigma(M_2)-\phi_\sigma(M_1)
\]
for some indecomposable objects $M_1, M_2$.
Then $M_1$ is not the shift of $M_2$ as $y\notin\ZZ$.
By Corollary~\ref{cor:bb}, we deduce that
$M_1, M_2$ correspond to the graded curves $\widetilde{\eta_1}, \widetilde{\eta_2}$
which connect marked points and only intersect at marked points/endpoints.
This implies that $y\le \angle_p\core(\sigma)$ for some $p\in \eta_1\cap \eta_2 \subset \Y$.
Thus, we have
\[
    \gldim\sigma\le \max\angle\core(\xi)
\]
and clearly the $\max\angle\core(\xi)$ is reachable.

Finally, the condition $M_1[\ZZ]\ne M_2[\ZZ]$ in \eqref{eq:CP} says that $M_1$ is not a shift of $M_2$
and the deduction above also implies that when $x=\phi_\sigma(M_2)-\phi_\sigma(M_1)$,
the corresponding curves $\eta_i$ intersect at a point in $\Y$.
Thus this pair corresponds to an angle of $\core(\xi)$.
\end{proof}

\begin{proposition}\label{pp:reach2}
Suppose that $\gldim\sigma\le1$.
Then $\sigma$ is $\gldim$-reachable and \eqref{eq:reach1} still holds.
\end{proposition}
\begin{proof}
If $\gldim\sigma<1$,
Theorem~\ref{thm:KOT} says that this happens if and only if $\DS$ is
of the form $\DQ/\iota$ for a Dynkin quiver $Q$.
This will force $\surf$ being a disk and $Q$ being an $A_n$ quiver (with $\iota=\id$).
Then $\sigma$ is $\gldim$-reachable due to finiteness of the category.

Next, consider the case $\gldim\sigma=1$ and we can exclude the disk case as above.
By \cite[Prop.~3.5]{Q1}, $\sigma$ is totally semistable, i.e. any indecomposable object $M$ is $\sigma$-semistable.
Take any boundary component $\partial_0$ with winding number $w$
and let $p\in\partial_0\cap\Y(\xi)$. So there is a loop $\gamma_k$ (non-trivial since $\gms$ is not a disk), for any $k\ge1$,
based at $p$ and go around $\partial_0$ for $k$ times, with a self-intersection of index $1+kw$.
See the orange loop for $\gamma_1$ in Figure~\ref{fig:wind around}, where $\partial_0$ is the outer boundary.
Let $M_k$ be the indecomposable object corresponding to some graded version of $\gamma_k$.
So by \eqref{eq:int} we have $\Hom(M_k,M_k[1+kw])\neq0$.
As $M_k$ is semistable, we have $1+kw\ge 0$ and hence $w\ge0$.
But $\gldim=1$ forces $1+kw\le1$, i.e. $w\le0$.
Thus, $w=0$ and $\gldim$ is achieved by $M_1$ and $M_1[1]$ corresponding to an angle at $p$.
Note that in this case $\gldim$ is also achieved by a $\kk\mathbb{P}^1$-family of objects.
\end{proof}

Combing the propositions above,
we know that any stability condition on $\DS$ is $\gldim$-reachable.
\begin{corollary}\label{cor:reach2}
Any $\sigma\in\Stab\DS$ is $\gldim$-reachable.
\end{corollary}
\subsection{Cycles of saddle connections and critical values}

\begin{proposition}\label{pp:min arc}
Suppose that $\gldim\sigma$ is reached at $p_1$ and $p_2$, i.e.
\[
    \gldim\sigma=\angle_{p_j}\core(\sigma)=\phi_\sigma(M_+^{p_j})-\phi_\sigma(M_-^{p_j}),\quad j=1,2.
\]
Let $\eta_\pm^{p_j}$ be the arcs corresponding to $M_\pm^{p_j}$.
If $\eta_+^{p_1}=\eta_-^{p_{2}}$,
then $\eta_+^{p_1}$ is a minimal arc (cf. Section~\ref{sec:Apq}).
\end{proposition}
\begin{proof}
Consider the arc $\eta=\eta_+^{p_1}=\eta_-^{p_{2}}$.
Any geodesics starting from $p_1$ that is on the right hand side (clockwise side with respect to $p_1$)
can not end at a closed marked point (infinity order zero) since $\eta_+^{p_1}$ is the rightmost saddle connection.
Therefore, they can only end at the boundary where $p_1$ lives.
Similarly, any geodesics starting from $p_{2}$ that is on the left hand side (anticlockwise side with respect to $p_2$)
can not end at a closed marked point.
Hence, they can only end at the boundary where $p_{2}$ lives.

Take all horizonal strips that intersect $\eta$.
They must have finite height and the saddle connections of these strips
form a broken geodesic connecting $p_1$ and $p_2$ (dashed line segment in Figure~\ref{fig:strip}).
They will be on the right hand side of $\eta$ when walking from $p_2$ to $p_1$.
Therefore on the other/left hand side, the infinities of these strips tend
to an infinity order pole/open marked point on some boundary of $\surf$.
By the discussion above, such an open marked point is the boundary where both $p_{1}$ and $p_{2}$ live.
See Figure~\ref{fig:strip}.
\begin{figure}[h]\centering
\begin{tikzpicture}[rotate=180-144-18,xscale=-.8,yscale=.8,arrow/.style={->,>=stealth,thick}]
\foreach \j in {2}
    {\draw[very thick]plot [smooth,tension=1.2] coordinates { (72*\j+105:3.5)(72*\j+90:3)(72*\j+75:4.5)};
    \draw(-72*\j+90-144:3.5)node[black]{$p_\j$};}
\foreach \j in {1}
    {\draw[very thick]plot [smooth,tension=1.2] coordinates { (72*\j+105:4.5)(72*\j+90:3)(72*\j+75:3.5)};
    \draw(-72*\j+90-144:3.5)node[black]{$p_\j$};}
\draw[red,thick,dashed, thick](72*2+18:3)\ww to (0,0)\ww to(-1,-1)\ww to (72*2+90:3)\ww;
\foreach \j in {-20,-15,...,25}
    {\draw[cyan!30,thick](72*2+54+180+45:3)to[bend left=\j](72*2+54:4.5);
    \draw[cyan!30,thick](0,-2)to[bend left=-\j](72*2+54:4.5);}
\foreach \j in {1,2,3}{
    \draw[red,thick] (72*\j+18:3)\ww to[bend left=0] (72*\j+90:3)\ww;}
\draw[white,fill=white](72*2+54+180+45:3)circle(.2)(72*2+54:4.5)circle(.2)(0,-2)circle(.2);\draw[cyan](72*2+54:4.5)\nn;
\draw[red](72*2+54:2.5)node[above right]{\footnotesize{$\eta_+^{p_1}=\eta_-^{p_2}$}};

\draw[red](72+55:2.8)node{\footnotesize{$\eta_+^{p_2}$}};
\draw[red](72*3+55:2.8)node{\footnotesize{$\eta_-^{p_1}$}};
\end{tikzpicture}
\caption{Horizontal strips containing the saddle connection $\eta$}\label{fig:strip}
\end{figure}
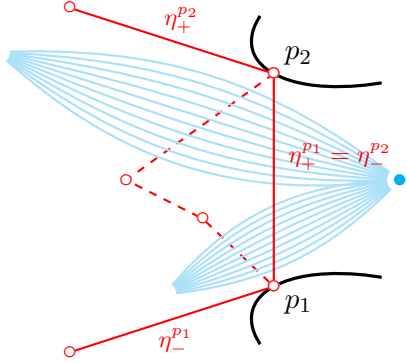
Thus we have shown that $p_1$ and $p_2$ are in the same boundary $\partial$ of $\surf$
and they are adjacent closed marked points.
\end{proof}

\begin{figure}[h]\centering
\begin{tikzpicture}[xscale=-.8,yscale=.8,arrow/.style={->,>=stealth,thick}]
\clip(0,0) circle (3.8);
\foreach \j in {1,2,3,4,5}
    {\draw[very thick]plot [smooth,tension=1.2] coordinates { (72*\j+105:3.5)(72*\j+90:3)(72*\j+75:3.5)};
    \draw[red,thick](-72*\j+90:3.5)node[black]{$p_\j$} (72*\j+18:3)\ww to[bend left=10] (72*\j+90:3)\ww
        (72*\j+54:2.9)node[rotate=72*\j-36]{$=$}
        (-72*\j+66:2.9)node{$\eta_+^{p_\j}$}
        (-72*\j+114:2.9)node{$\eta_-^{p_\j}$};}
\draw(0,0)node{$\cdots$};
\end{tikzpicture}
\begin{tikzpicture}[xscale=-.8,yscale=.8,arrow/.style={->,>=stealth,thick}]
\clip(0,0) circle (3.8);
\draw[very thick,fill=gray!11] (0,0) circle (2) (3.7,0)node{\huge{=}};
\foreach \j in {1,2,3,4,5}
    {\draw[red,thick]plot [smooth,tension=1.2] coordinates { (72*\j:2) (72*\j+36:2.6) (72*\j+72:2) }
        (72*\j+90+36:1.7) node[black]{\footnotesize{$p_{\j}$}}
        (72*\j:2)\ww(72*\j+72:2)\ww (72*\j+36+144:3)node{$\eta_+^{p_\j}$};
    }
\end{tikzpicture}
\caption{A cycle of saddle connections}\label{fig:circle}
\end{figure}
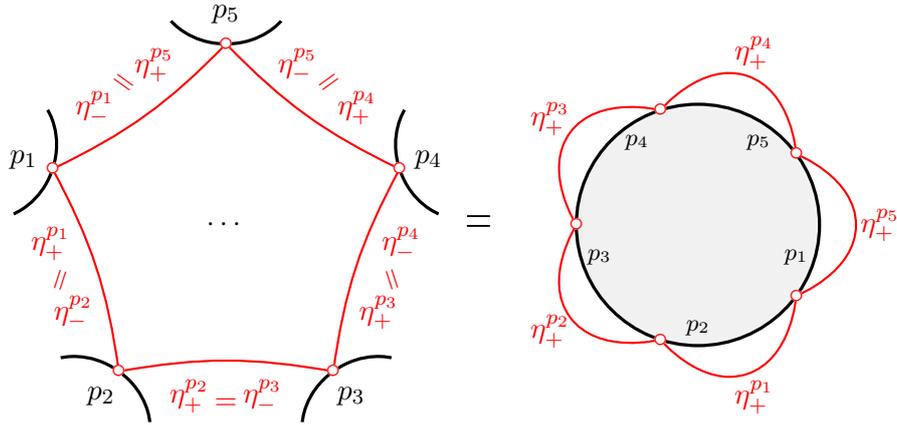
\begin{corollary}\label{cor:cycle}
Suppose that $\gldim\sigma$ is reached at $p_1,p_2,\ldots,p_m$, i.e.
\[
    \gldim\sigma=\angle_{p_j}\core(\sigma)=\phi_\sigma(M_+^{p_j})-\phi_\sigma(M_-^{p_j}),\quad j=1,\ldots,m.
\]
Let $\eta_\pm^{p_j}$ be the arcs corresponding to $M_\pm^{p_j}$.
If $\eta_+^{p_j}=\eta_-^{p_{j+1}}$ for $1\le j\le m$ and $\eta_+^{p_m}=\eta_-^{p_{1}}$,
cf. the left picture in Figure~\ref{fig:circle},
then the arcs $\eta_\pm^{p_j}$ are precisely all the minimal arcs
(cf. Section~\ref{sec:Apq})
at some boundary $\partial$ of $\surf$ in clockwise order (cf. the right picture in Figure~\ref{fig:circle}).
Note that in such a case we will have
\begin{gather}\label{eq:1+w/m}
    \gldim\sigma=1+w_\partial/m_\partial,
\end{gather}
for $m=m_\partial$ the number of marked point on $\partial$
and $w_\partial$ the winding number of $\partial$.
\end{corollary}
\begin{proof}
By repeatedly using Proposition~\ref{pp:min arc} above we end up as the right picture of Figure~\ref{fig:circle}.

For the final calculation, we only need to notice that, by properly shifting $M_\pm^{p_j}$,
we can arrange that
\[
    M_+^{p_j}=M_-^{p_{j+1}},\quad 1\le j\le m-1
\]
and then $M_+^{p_m}=M_-^{p_{1}}[t]$, where $t$ can be calculated as in \eqref{eq:index sum},
that equals $w+m$.
So we have
\[
    m\cdot\gldim\sigma=\sum_{j=1}^m \phi_\sigma(M_+^{p_j})-\phi_\sigma(M_-^{p_j})=w+m
\]
as claimed.
\end{proof}

Note that in the situation of the proposition above, we have $w\ge0$ unless $\surf$ is a disk.
This follows from the fact that $\gldim\sigma\ge1$ unless $\surf$ is a disk (of type A).

Denote by
\begin{gather}
    \VC(\gms)=\{1+w_\partial/m_\partial \mid \partial\subset\partial\surf, w_\partial\ge0\}
\end{gather}
be the set of critical values of $\gldim$.

We can upgrade the second statement of Proposition~\ref{pp:reach1} a little bit.
\begin{proposition}\label{pp:reach1+}
If $\gldim\sigma>1$, then $\sigma$ is $\gldim$-reachable exactly by
the pair $(M_1,M_2)$ of objects corresponding to two edges of an angle of $\core(\xi)$.
Moreover,  $M_1[\ZZ]=M_2[\ZZ]$ can only happen if $\gldim\sigma=1+w_\partial/m_\partial$ for some $\partial$ with
$m_\partial=1$.
If this does not happen, then \eqref{eq:CP} consists of precisely all such pairs $(M_1,M_2)$.
\end{proposition}
\begin{proof}
This follows from the fact that different arcs corresponds to different objects.
So $M_1[\ZZ]=M_2[\ZZ]$ implies that two edges of an angle of $\core(\xi)$ coincide.
And Corollary~\ref{cor:cycle} implies that they bound a boundary with exactly one closed marked point.
\end{proof}

\subsection{Main result}
Recall that for $\sigma=(Z,\hh{P})$, we define a set $\CP{\sigma}$ of pairs of objects in \eqref{eq:CP},
whose phase difference of each pair reaches $\gldim\sigma$.
Let $\uCP{\sigma}=\CP{\sigma}/[1]$ be the set of shift orbits of such pairs.

\begin{theorem}\label{thm:main}
If $1\le\gldim\sigma\notin\VC(\gms)$,
then $\Stab_{\CP{\sigma}}\DS$ consists of an real submanifold $\Stab_{\CP{\sigma}}\DS$ of $\Stab\DS$ with
\[
    3\le\dim_\RR \Stab_{\CP{\sigma}}\DS=2n+1-s,
\]
for $s=\#\uCP{\sigma}$.
Moreover, $\Stab_{\CP{\sigma}}\DS$ is open in its closure and,
restricted to which, $\gldim$ is differentiable without critical point.
\end{theorem}
\begin{proof}
By Proposition~\ref{pp:reach1} and ~\ref{pp:reach2},
we know that $\gldim\sigma$ will be only reached at certain closed marked points $p_1,p_2,\ldots,p_s$,
in the sense that
\begin{equation}\label{eq:eqs}
    \gldim\sigma=\angle_{p_j}\core(\sigma)=\phi_\sigma(M_+^{p_j})-\phi_\sigma(M_-^{p_j}),\quad j=1,\ldots,s.
\end{equation}
Thus, we have
\[
    \CP{\sigma}=\big\{ (M_\pm^{p_j}[m]) \mid m\in\ZZ, 1\le j\le s \big\}.
\]
Up to the $\CC$-action, we can assume that the heart $\h_\sigma$ of $\sigma$ is finite/algebraic
(i.e. a length category with finitely many simples).
Then (cf. \cite[Lem.~5.2]{B1} and \cite{QW})
$\sigma$ is the half-open-half-closed cube $U(\h)\cong\mathbf{H}^n\subset\Stab\DS$
(recall that $\mathbf{H}$ is the upper half plane),
where the coordinates are given by the central charges $\mathbf{Z}_\sigma=\{Z(S_i)\}$ of simples $S_i$ in $\h_\sigma$.

Let $m_\pm\in\ZZ$ such that $M_\pm^{p_j}[m_\pm]$ is in $\h_\sigma$.
Then $Z_j^\pm\colon=Z(M_\pm^{p_j}[m_\pm])$ will be the linear combinations of central charges of simples in $\mathbf{Z}_\sigma$.
Let $G^\sigma$ be the directed graph whose vertices are $\{ M_\pm^{p_j}[m_\pm] \mid 1\le j\le s \}$
and whose arrows are $$\{ M_-^{p_j}[m_-]\to M_+^{p_j}[m_+] \mid 1\le j\le s \}.$$
As $\gldim\sigma\notin\VC(\gms)$, Corollary~\ref{cor:cycle} implies that
there is no cycle in $G^\sigma$.
In fact, any connected component of $G^\sigma$ has the following form
\[
\begin{tikzpicture}[yscale=.5]
\draw(0,0) node (x1) {} (1,.5) node (x2) {} (2,1) node (x3) {}  (4.8,2.4) node (y1){};
\draw(0,0) node (x1) {} (1,-.5) node (x4) {} (2,-1) node (x5) {} (4.8,-2.4) node (y2){};
\foreach \j/\i in {1/2,2/3,1/4,4/5}{
\draw[->,>=stealth](x\j)\ww to (x\i) \ww ;}
\draw[->,>=stealth](x3) to (3,1.5) node[right,rotate=17]{$\cdots$};
\draw[->,>=stealth](x5) to (3,-1.5) node[right,rotate=-17]{$\cdots$};
\draw[->,>=stealth](3.8,1.9) to  (y1)\ww;
\draw[->,>=stealth](3.8,-1.9) to  (y2)\ww;
\draw(3,3.5) node[right,rotate=-17]{$\cdots$};
\draw(3,-3.5) node[right,rotate=17]{$\cdots$};
\draw[->,>=stealth](3.8,2.9) to  (y1)\ww;
\draw[->,>=stealth](3.8,-2.9) to  (y2)\ww;
\end{tikzpicture}.
\]
Moreover, Proposition~\ref{pp:min arc} can be translated to: if some $M_\pm^{p_j}[m_\pm]$
is neither a source nor a sink in $G^\sigma$, then it corresponds to a minimal arc.
A consequence is that the ungraded arcs $\{ \eta_\pm^{p_j} \mid 1\le j\le s \}$
can be completed to a full formal arc system.
Therefore, in the Grotendieck group $$K\DS=\< [S_i] \mid \text{simple $S_i$ in $\h_\sigma$} \>,$$ the classes of $M_\pm^{p_j}[m_\pm]$ form a partial basis.
Since the central charge $Z$ is a group homomorphism,
$\{Z_j^\pm \mid 1\le j\le s\}$ are linear independent in the coordinate $\mathbf{Z}_\sigma$.
Furthermore, the no-cycle condition in $G^\sigma$ implies that
the differences $\{Z_j^+-Z_j^- \mid 1\le j\le s\}$ are also linear independent in the coordinate $\mathbf{Z}_\sigma$.
Thus, by change of coordinates, we can choose
$\{Z_j^+-Z_j^- \mid 1\le j\le s\}$, together with some $Z(S_i)$ (or their linear combinations),
to be the coordinates in the neighbourhood of $U(\sigma)$ of $\sigma$,
where we use polar coordinate system $z=m\cdot e^{\bi \pi \theta}$ for complexes ($m\in\RR_+, \theta\in\RR$)
regarding $\Stab$ as a real manifold.

Next, we claim that there is a neighbourhood $U(\sigma)$ of $\sigma$ in $\Stab\DS$, so that
\begin{gather}\label{eq:=max}
    \gldim\varsigma=\max \{ \angle_{p_j}\core(\varsigma) \mid 1\le j\le s \}
\end{gather}
for any $\varsigma\in U(\sigma)$.
To see this, let
\[
    \epsilon=\gldim\sigma-\max \{ \angle_p\core(\sigma) \mid p\ne p_j, 1\le j\le s \},
\]
so that for any other pair of $\sigma$-semistable indecomposable objects $(M_1',M_2')\notin\CP{\sigma}$
with $\Hom(M_1',M_2')\neq0$ and $M_1'[\ZZ]\neq M_2'[\ZZ]$, we have
\[
    \phi_\sigma(M_2')-\phi_\sigma(M_1') \le \gldim-\epsilon.
\]
Take $U(\sigma)$ be the open ball with center $\sigma$ and radius $\epsilon/4$ and recall that
the distance on $\Stab$ is defined by
\begin{equation}
\label{eq:distance}
d(\sigma,\varsigma):= \sup_{0 \neq E \in \D}\left\{\,
|\phi_{\sigma}^-(E) - \phi_{\varsigma}^-(E)|\,,\,
|\phi_{\sigma}^+(E) - \phi_{\varsigma}^+(E)|\,,\,
\left|\log \frac{m_{\sigma}(E)}{m_{\varsigma}(E)}  \right|
\right\}.
\end{equation}
Then for any $\varsigma=(W,\hh{Q})\in U(\sigma)$,
we have
$$\hh{Q}(\varphi)\subset\hh{P}(\varphi-\epsilon/4, \varphi+\epsilon/4),\quad\forall\varphi\in\RR,$$
where $\sigma=(Z,\hh{P})$.
Then we deduce that for any pair $(M_1',M_2')\notin\CP{\sigma}$ as above,
we will have
\[
    \phi_\varsigma(M_2')-\phi_\varsigma(M_1')\le \phi_\sigma(M_2')-\phi_\sigma(M_1') + \epsilon/2.
\]
Similarly, $\gldim\varsigma>\gldim\sigma-\epsilon/4$ (cf. \cite{IQ1})
which implies the claim.

In a neighbourhood of $\sigma$ in $\Stab\DS$,
consider a submanifold $\hh{V}$ defined by the equations \eqref{eq:eqs},
or equivalently
\[ \theta_1=\theta_2=\cdots=\theta_s \]
for $Z_j^+-Z_j^-=m_j\cdot e^{\bi \pi \theta_j}$.
It is a real submanifold in $\Stab\DS$ with dimension $2n-(s-1)$ and we have
\[
     U_0(\sigma)\subset \Big(\hh{V}\cap\Stab_{\CP{\sigma}}\DS \Big)
\]
for any small neighbourhood $U_0(\sigma)$ of $\sigma$ in $\Stab_{\CP{\sigma}}\DS$.
Let $\varsigma\in \hh{V}\cap\Stab_{\CP{\sigma}}\DS$.
Since the objects $M_\pm^{p_j}$ that appears in $\CP{\sigma}$ is stable,
there is a neighbourhood $U_0(\varsigma)$ of $\varsigma$ in $\Stab\DS$, such that $M_\pm^{p_j}$ remains stable (before it gets destabilized to be semistable).
Then as \eqref{eq:=max} holds and $\CP{\varsigma}$ remains unchange,
we have
\[
    \Big( U_0(\varsigma)\cap \hh{V}  \Big)\subset \Stab_{\CP{\sigma}}\DS,
\]
that implies $\Stab_{\CP{\sigma}}\DS$ is open locally (in $\hh{V}$) and hence open in its closure as claimed.

Furthermore, when restricted to this submanifold, $\gldim$ is in fact given by a single coordinate.
Thus, it is differentiable without critical point.

Finally, we estimate the real dimension of $\Stab_{\CP{\sigma}}\DS$.
Semi-stable objects and their phase differences are invariant under the $\CC$-action.
Thus, $\Stab_{\CP{\sigma}}\DS$ is closed under the $\CC$-action,
which implies that its real dimension is at least two.
Moreover, $\gldim$ is invariant under the $\CC$-action.
Together with the fact that $\gldim$ has no critical point as we showed above,
we deduce that $\dim_\RR \Stab_{\CP{\sigma}}\DS\ge3$.
Another way to see this is via a direct calculation.
Namely, we have (from \eqref{eq:rank})
\[
    2n+1-s=4g+2b+2\aleph-3-s.
\]
If $g\ge2$ or $b\ge2$, we have $\aleph\ge b\ge 1$ and $\aleph\ge s$, which implies $2n+1-s\ge3$.
If $g=0, b=1$, then $n\ge2$ implies $\aleph\ge3$.
We claim that $s\le \aleph-1$. Otherwise
there is a cycle of saddle connections as in Figure~\ref{fig:circle},
such that they correspond to a collection of semistable objects
whose phase difference is $\gldim$.
But the winding number $w$ is $-2$, which implies \eqref{eq:1+w/m} for $m=\aleph$.
This contradicts to $\gldim\ge1$.
So we always have $2n+1-s\ge3$.
\end{proof}

We proceed to analyze what happens on the boundary of subspaces $\Stab_{\CP{\sigma}}\DS$.
\begin{corollary}\label{cor:main}
If $1\le y<x$ such that $(y,x)\cap\VC(\gms)=\emptyset$,
then $\Stab_{\le x}\DS$ contracts to $\Stab_{\le y}\DS$ via a flow induced by the differential of $\gldim$.
\end{corollary}
\begin{proof}
By Theorem~\ref{thm:main}, if $\sigma$ is in $\Stab_{\CP{\sigma}}\DS$,
then $\diff \gldim$ gives a contractible flow.
When the flow hits the boundary of $\Stab_{\CP{\sigma}}\DS$,
it necessarily enters another $\Stab_{\CP{\varsigma}}\DS$.
Then the corollary follows.
\end{proof}

\begin{remark}
When following the contractible flow in $\Stab_{\CP{\sigma}}\DS$ and hits its boundary and then enters another $\Stab_{\CP{\varsigma}}\DS$ at $\varsigma$, typical scenarios are
\begin{itemize}
  \item $\CP{\sigma}$ gets bigger that $\gldim$ is achieved at another angle of the core or;
  \item some of the $M_\pm^{p_j}$ get destabilized to be semistable and the corresponding pairs in $\CP{\sigma}$
  are replaced.
\end{itemize}
In the next section, we will examine type A and (graded) affine type A cases in more details to show such phenomenons.
\end{remark}

Another immediate consequence is the following.
\begin{corollary}\label{cor:gd}
If $\gms$ is not a disk, then $\gd\DS$ is in $\VC(\gms)$.
\end{corollary}

\section{Examples}\label{sec:ex}
Denote by
$
    \PStab(-)=\Stab(-)/\CC
$
the spaces of projective stability conditions,
where $\gldim$ is well-defined.

\subsection{Rank 2 cases and deformation}
\begin{example}\label{ex:A2}
Consider the case when $\surf$ is a disk with three marked points,
where $$\DS\cong\D^b(\kk A_2)$$ is the bounded derived category of an $A_2$ quiver $1\to2$.
Let $Z_1=Z(S_1)=Z(P_1),Z_2=Z(S_2)$ and $Z_3=Z(P_2[1])$.
Then $\PStab\DS$ decomposes into:
\begin{itemize}
\item
    three $1$-$\dim_\RR$ subspaces (blue lines in Figure~\ref{fig:logo}),
    which correspond to equations
    \[ |Z_i|=|Z_j|,\quad \{i,j\}\in\{1,2,3\} ;\]
\item
    three $2$-$\dim_\RR$ subspaces green areas in Figure~\ref{fig:logo}
    that are bounded by the $1$ subspaces above;
\item
    one critical point $\CC\cdot\sigma_G$ with $|Z_1|=|Z_2|=|Z_3|$,
    which is the solution in Theorem~\ref{thm:Q}.
\end{itemize}
The contractible flow is shown in Figure~\ref{fig:logo};
\begin{figure}[h]\centering
\begin{tikzpicture}[scale=.6] \clip(0,0) circle (4.2);
\draw[white,fill=Emerald!10](0,0) circle (4.2);
\foreach \k in {0,1,2}
{ \path (120*\k:5) coordinate (v2)
          (120*\k+60:5) coordinate (v1)
          (120*\k+120:5) coordinate (v3)
          (0,0) coordinate (v4);
  \foreach \j in {.05,.15,...,.95}
    {
      \draw[Emerald,->-=.5,>=stealth] ($(v3)!\j!(v1)$)to($(v3)!\j!(v4)$);
      \draw[Emerald,->-=.5,>=stealth] ($(v2)!\j!(v1)$)to($(v2)!\j!(v4)$);
    }
  \draw[very thick,blue,->-=.5,>=stealth](v2)to(v4);
}
\draw[very thick,blue](5,0)to(0,0);
\draw[red,fill=red](0,0)circle(.1);
\end{tikzpicture}
\caption{Contractible flow in $A_2$ case}\label{fig:logo}
\end{figure}
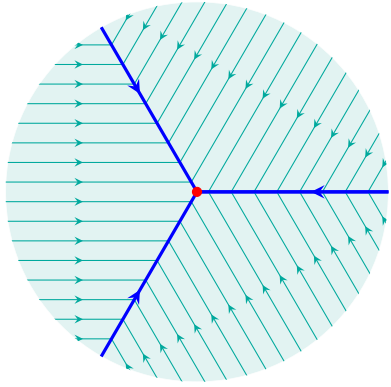
\end{example}

\begin{example}\label{ex:K2}
Consider the case when $\surf$ is an annulus with one marked point on each boundary,
where $$\DS\cong\D_\infty(\kk K_2)\cong\D^b(\coh\bP^1)$$ is the bounded derived category of the Kronecker quiver $K_2\colon 1\rightrightarrows2$ or the one of coherent sheaves on $\coh\bP^1$.
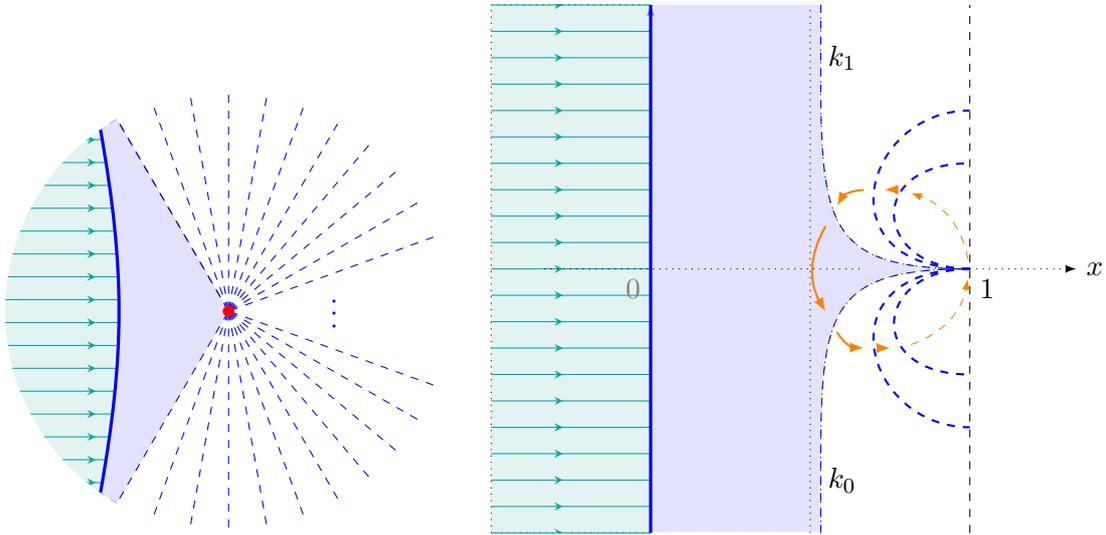
\begin{figure}[h]\centering
\begin{tikzpicture}[scale=.7] \clip(0,0) circle (4.2);
\draw[white,fill=Emerald!10](0,0)to(120:7)to[bend right=60](240:7);
\foreach \k in {1}
{ \path (120*\k:5) coordinate (v2)
          (120*\k+60:5) coordinate (v1)
          (120*\k+120:5) coordinate (v3)
          (0,0) coordinate (v4);
  \foreach \j in {.05,.15,...,.95}
    {
      \draw[Emerald,->-=.5,>=stealth] ($(v3)!\j!(v1)$)to($(v3)!\j!(v4)$);
      \draw[Emerald,->-=.5,>=stealth] ($(v2)!\j!(v1)$)to($(v2)!\j!(v4)$);
    }
}
\draw[blue,dashed,fill=blue!11](0,0)to(120:4.21)to(125:4.21)to[bend left=10](235:4.21)to(240:4.21);
\draw[dashed,blue](240:5)to(0,0);
\draw[dashed,blue](120:5)to(0,0);
\draw[blue,very thick](125:4.2)to[bend left=10](235:4.2);
\draw[thick,blue](2,0)node[rotate=90]{$\cdots$};
\foreach \j in {-110,-100,...,-20}{\draw[dashed,blue](\j:5)to(0,0);}
\foreach \j in {110,100,...,20}{\draw[dashed,blue](\j:5)to(0,0);}
\draw[red,fill=red](0,0)circle(.1);
\end{tikzpicture}
\quad
\begin{tikzpicture}[scale=.7]\clip(-3,-5) rectangle (8.5,5);

\path (0,0) coordinate (O);
\path (4,0) coordinate (A);
\path (3+0.2,5) coordinate (m3);
\path (3+0.2,-5) coordinate (m4);
\draw[fill=blue!11,dotted]   (m3)
    .. controls +(-90:4) and +(180:3) .. (6,0)
    .. controls +(180:3) and +(90:4) .. (m4)
    to [out=180,in=0] (-3,-5)
    to [out=90,in=-90] (-3,5);
\draw[white, thick]   (m4) to (-3,-5) to (-3,5);

\draw[fill=Emerald!10,dotted]   (-3,-5)rectangle(0,5);

\foreach \j in {-5,-4.5,...,5}
{\draw[Emerald,->-=.45,>=stealth] (-3,\j)to(0,\j);}

\draw[blue,dashed] (m3) .. controls +(-90:4) and +(180:3) .. (6,0);
\draw[blue,dashed] (m4) .. controls +(90:4) and +(180:3) .. (6,0);

\draw[dotted,->,>=latex] (6,0) -- (8,0) node[right]{$x$};
\draw[gray,dashed,->,>=latex] (0,-5) -- (0,5) node[above]{$y$};
\draw[blue, very thick] (0,-5) -- (0,5);
\draw[gray] (0,0) node[below left]{$0$};
\path[dotted] (-2,0) edge (4,0);
\path[dashed] (6,-5) edge (6,5);
\path (3,-5)[dotted] edge (3,5);

\path (3.6,4) node {$k_1$};\path (3.6,-4) node {$k_0$};
\draw[blue,dashed,thick] (6,0) .. controls +(180:2.8) and +(180:2) .. (6,3);
\draw[blue,dashed,thick] (6,0) .. controls +(180:2.8) and +(180:2) .. (6,-3);
\draw[blue,dashed,thick] (6,0) .. controls +(180:1.8) and +(180:2) .. (6,2);
\draw[blue,dashed,thick] (6,0) .. controls +(180:1.8) and +(180:2) .. (6,-2);
\path (6,0) node[below right] {$1$};
\path (6,0) node (a1) {\tiny{}};
\path (4.7,1.5) node (a2) {\tiny};
\path (4.2,1.5) node (a3) {\tiny};
\path (3.4,1) node (a4) {\tiny};
\path (4.7,-1.5) node (a7) {\tiny};
\path (4.2,-1.5) node (a6) {\tiny};
\path (3.4,-1) node (a5) {\tiny};
\path[->,orange, bend right, dashed] (a1) edge (a2);
\path[->,orange, thick] (a2) edge (a3);
\path[->,orange, bend right, thick] (a3) edge (a4);
\path[->,orange, bend right, thick] (a4) edge (a5);
\path[->,orange, bend right, thick] (a5) edge (a6);
\path[->,orange, thick] (a6) edge (a7);
\path[->,orange, bend right, dashed] (a7) edge (a1);
\end{tikzpicture}
\caption{Contractible flow in the Kronecker case}\label{fig:logo2}
\end{figure}
Similarly to the $A_2$ case,
$\PStab\DS$ decomposes into (cf. \cite{O} and \cite[\S~7.5.2]{Q2}):
\begin{itemize}
\item
    a core $2$-$\dim_\RR$ subspaces $\PStab_{=1}\D^b(\coh\bP^1)$;
\item
    $\ZZ$ many copies of $1$-$\dim_\RR$ subspaces (dashed blue lines in Figure~\ref{fig:logo2})
    in $\PStab_{=1}\D^b(\coh\bP^1)$,
    which correspond to equations
    \[ |Z(\hh{O}(j-1)[1])|=|Z(\hh{O}(j))|,\quad j\in\ZZ. \]
    They are related by $-\otimes \hh{O}(1)\in\Aut\D^b(\coh\bP^1)$;
\item
    $\ZZ$ many copies of $2$-$\dim_\RR$ subspaces green areas in Figure~\ref{fig:logo},
    each of which is a connected component of
    \begin{gather}\label{eq:zz}
        \PStab_{>1}\D^b(\coh\bP^1).
    \end{gather}
    They are also related by $-\otimes \hh{O}(1)\in\Aut\D^b(\coh\bP^1)$.
\end{itemize}
The contractible flow is shown in Figure~\ref{fig:logo2} (exists in \eqref{eq:zz});
\end{example}
\subsection{Disk case revisit}
Let $\surf$ be a disk with $n+1$ marked points, i.e.
$g=0, b=1, \aleph=n+1$ and $w=-2$.
Then $\DS\cong\DA$ for an $A_n$ quiver.
Recall that $\Stab_{<1}\DA$, consists of all totally stable stability conditions in this case.
Its projective version is isomorphic to the space of convex $(n+1)$-gon (Proposition~\ref{pp:A}).

In this case, $\VC(\gms)=\{1,(n-1)/(n+1)\}$,
where $(n-1)/(n+1)$ is in fact $\gd\DA$.
Then Theorem~\ref{thm:main} can be rephrased as following corollary.

\begin{corollary}\label{cor:A}
$\Stab_{<y}\DA$ contracts to $\Stab_{<x}\DA$ for any $1\le x\le y$.
In particular, $\Stab\DA$ contracts to $\Stab_{<1}\DA$,
\end{corollary}
\begin{proof}
We only need to show the second statement.
On one hand, we have $$\Stab\DA=\lim_{y\to\infty}\Stab_{<y}\DA.$$
On the other hand, any $\sigma\in\Stab\DA$ with $\gldim\sigma=1$ is in some (open) real submanifold
$\Stab_{\CP{\sigma}}\DA$ and thus can be further contracted.
Thus, the statement follows.
\end{proof}

\subsection{Annulus case revisit}\label{sec:Apq}
We keep the notation in Example~\ref{ex:Apq}, i.e.
we have $\DS\cong\Dpq$ and $\num(\gms)=((m,r),(w,-w))$ for $m,r\in\ZZ_+$ and $w\in\ZZ_{\ge0}$.
Then $\VC(\gms)=\{1,1+w/m\}$.

\begin{theorem}\label{thm:Apq}
$\Stab_{<y}\Dpq$ contracts to $\Stab_{<x}\Dpq$ for any $1+w/m\le x\le y$.
Moreover,
$\gd\Dpq=1+w/m$ and in particular $\Stab\Dpq$ contracts to $\Stab_{1+w/m}\Dpq$.
\end{theorem}
\begin{figure}[h]\centering
\begin{tikzpicture}[xscale=.7,yscale=.7,arrow/.style={->,>=stealth,thick}]
\begin{scope}[shift={(0,0)}]
\draw[Green,very thick,->-=.55,>=stealth]
        (-.6,.95)to[bend left=75](.6,.95);
\draw[thick,fill=gray!11] (0,0) circle (1);\draw[thick](0,0) circle (3);
\draw[red, very thick]
    (0,1).. controls +(0:3) and +(0:1) ..(0,-1.4)
    (0,1).. controls +(180:3) and +(180:1) ..(0,-1.4)
    (0,-1.5)node[below]{$\eta$};
\foreach \j in {1,2,3} {\draw[red,thick,dashed] (120*\j+90:3)\ww (120*\j+90:1)\ww ;}
\end{scope}
\begin{scope}[shift={(7,0)}]

\draw[thick,fill=gray!11] (0,0) circle (1);\draw[thick](0,0) circle (3);

\draw[red, very thick]
    (0,1).. controls +(0:3) and +(10:1) ..(0,-1.4)
    .. controls +(-10:3) and +(0:3) ..(0,1.8)
    (0,1).. controls +(180:3) and +(170:1) ..(0,-1.4)
    .. controls +(-170:3) and +(180:3) ..(0,1.8)

    (0,-1.5)node[below]{$\gamma$};
\foreach \j in {1,2,3} {\draw[red,thick,dashed] (120*\j+90:3)\ww (120*\j+90:1)\ww ;}
\end{scope}
\end{tikzpicture}
\caption{An L-arc annulus case}\label{fig:ll}
\end{figure}
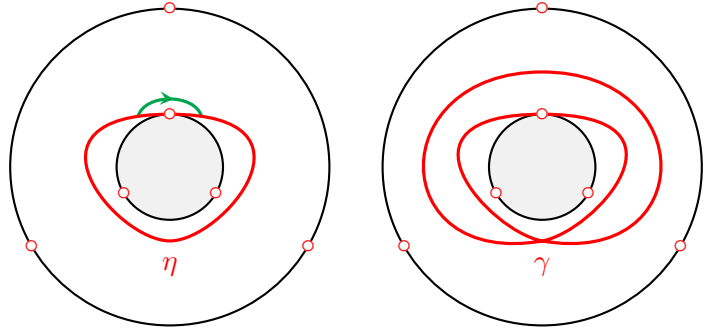
\begin{proof}
By Theorem~\ref{thm:main}, we only need to prove $\gd\Dpq=1+w/m$.
The $w=0$ case is contained in \cite[Thm.~5.2]{Q1}.

Consider the case when $w>0$ and suppose that $\gd\Dpq<1+w/m$.
We have $\Stab_{<1+w/m}\Dpq$ contracts to $\Stab_{\le1}\Dpq$ by Theorem~\ref{thm:main},
which implies that $\gd\Dpq\le1$.
By Corollary~\ref{cor:reach2},
this can only happen for $w=0$, which is a contradiction.
Thus, $\gd\Dpq\ge1+w/m$, which forces $\gd\Dpq=1+w/m$ by Lemma~\ref{lem:gd}.
\end{proof}


\end{document}